\documentclass[11pt]{article}
\usepackage[utf8]{inputenc}
\usepackage{authblk}
\title{Volumes of moduli spaces of hyperbolic surfaces with cone points}
\author{Lukas Anagnostou and Paul Norbury}
\affil{School of Mathematics and Statistics, University of Melbourne, VIC 3010, Australia}

\usepackage[numbers]{natbib}
\usepackage{graphicx}
\usepackage{amsmath}
\usepackage[utf8]{inputenc}
\usepackage{amssymb}
\usepackage{amsthm}
\usepackage{mathrsfs}
\usepackage{nccmath}
\usepackage{setspace}
\usepackage{enumitem}
\usepackage{physics}
\usepackage{mathtools}
\usepackage{tikz-cd}
\usetikzlibrary{arrows}
\usepackage{tikz}
\usepackage{chngcntr}
\usepackage{hang}
\usepackage[super]{nth}
\usepackage{xcolor}
\usepackage{bm}
\usepackage{xcolor}
\usepackage{hyperref}
\usepackage{enumitem}
\usepackage{comment}
\usepackage{appendix}

\usepackage{hyperref,cleveref}

\usepackage{newunicodechar}
\newunicodechar{よ}{\text{\usefont{U}{min}{m}{n}\symbol{'210}}}
\DeclareFontFamily{U}{min}{}
\DeclareFontShape{U}{min}{m}{n}{<-> udmj30}{}

\usepackage[T1]{fontenc}
\usepackage{tikz}\usetikzlibrary{decorations.markings, arrows, shapes, calc, 
quotes,angles}
\tikzset{>={Latex[width=0.25cm,length=0.25cm]}}
\definecolor{MyBlue}{HTML}{1d2bd1}
\definecolor{MyRed}{HTML}{e70000}
\definecolor{MyGreen}{HTML}{45ab02}
\definecolor{MyDarkGrey}{HTML}{828282}
\tikzset{->-/.style={decoration={
  markings,
  mark=at position 0.52 with {\arrow[#1, scale=1.1]{>}}},postaction={decorate}}}
\tikzset{->>-/.style={decoration={
  markings,
  mark=at position 0.52 with {\arrow[#1, scale=0.55]{>}}},postaction={decorate}}}
\definecolor{GGreen}{rgb}{0.0, 0.5, 0.0}

\usepackage{tkz-euclide}

\usetikzlibrary{cd}
\usepackage[a4paper, total={6in, 9in}]{geometry}
\allowdisplaybreaks

\usepackage{caption}
\theoremstyle{definition}
\newtheorem{definition}{Definition}[section]

\theoremstyle{plain}
\newtheorem{theorem}[definition]{Theorem}
\newtheorem{lemma}[definition]{Lemma}
\newtheorem{proposition}[definition]{Proposition}

\newtheorem{claim}[definition]{Claim}
\newtheorem{corollary}[definition]{Corollary}
\theoremstyle{remark}
\newtheorem{remark}{Remark}[section]

\newtheorem*{proof sketch}{Proof sketch}
\newtheorem*{proof of claim}{Proof of claim}

\newcommand{\bn}{\mathbb{N}}
\newcommand{\bq}{\mathbb{Q}}
\newcommand{\br}{\mathbb{R}}

\newcommand{\bz}{\mathbb{Z}}
\newcommand{\modm}{\mathcal{M}}

\newcommand{\smallsim}{\smallsym{\mathrel}{\sim}}

\usetikzlibrary{fit}
\tikzset{%
  highlight/.style={rectangle,rounded corners,fill=red!15,draw,
    fill opacity=0.5,thick,inner sep=0pt}
}

\makeatletter
\newcommand{\smallsym}[2]{#1{\mathpalette\make@small@sym{#2}}}
\newcommand{\make@small@sym}[2]{%
  \vcenter{\hbox{$\m@th\downgrade@style#1#2$}}%
}
\newcommand{\downgrade@style}[1]{%
  \ifx#1\displaystyle\scriptstyle\else
    \ifx#1\textstyle\scriptstyle\else
      \scriptscriptstyle
  \fi\fi
}
\makeatother

\DeclareMathOperator{\Hom}{Hom}

\DeclareMathOperator{\GL}{GL}

\DeclareMathOperator{\SL}{SL}
\DeclareMathOperator{\PSL}{PSL}

\DeclareMathOperator{\Out}{Out}
\DeclareMathOperator{\Inn}{Inn}
\DeclareMathOperator{\Aut}{Aut}

\newcommand\bbR{\mathbb{R}}

\newcommand\bbZ{\mathbb{Z}}

\newcommand\cD{\mathcal{D}}

\newcommand\cM{\mathcal{M}}

\newcommand\cT{\mathcal{T}}

\renewcommand{\b}{\boldsymbol}

  {\begin{list}{}%
          {\setlength{\leftmargin}{#1}}%
          \item[]%
  }
  {\end{list}}

\usepackage{color}   

\hypersetup{
    colorlinks,
    citecolor=red,
    filecolor=black,
    linkcolor=blue,
    urlcolor=black
}

\newcommand\gpl{g_{\mathbf{a}}}
\newcommand\wpl{\omega_{\mathbf{a}}}
\newcommand{\bp}{\mathbb{P}}

\begin{document}

\maketitle 

 \begin{abstract}
    In this paper we study volumes of moduli spaces of hyperbolic surfaces with cone angles. We compute the volumes in some new cases, in particular when there exists a large cone angle. This allows us to give geometric meaning to Mirzakhani's polynomials under substitution of imaginary valued boundary lengths, corresponding to hyperbolic cone angles, and to study the behaviour of the volume under the $2\pi$ limit of a cone angle.  
  \end{abstract}

\setcounter{tocdepth}{2}  
\tableofcontents

\section{Introduction}
Define the moduli space of hyperbolic surfaces with labeled boundary components
\begin{align*}
\modm_{g,n}^{\mathrm{hyp}}(L_1&,...,L_n)=\Big\{(\Sigma,\beta_1,...,\beta_n)\mid \Sigma \text{ oriented hyperbolic surface},\\
&\text{ genus }\Sigma=g, \partial \Sigma=\sqcup\beta_j, 2\cosh(L_j/2)=-\mathrm{tr}(A_j)\Big\}/\sim
\end{align*} 
for $g,n\in\bn=\{0,1,2,...\}$ satisfying $2g-2+n>0$, where $A_j\in PSL(2,\br)$ represents the conjugacy class defined by the holonomy of the metric around the boundary component $\beta_j$.  The quotient is by isometries preserving each $\beta_j$ and its associated conjugacy class.  Each $L_j\in\br_{\geq0}\cup i\hspace{.5mm} \br_{>0}$ 
and the boundary component is geometrically realised by a geodesic of length $L_j\in\br_{>0}$, or a cusp when $L_j=0$, or a cone angle $\theta_j=|L_j|\in\br_{>0}$ when $L_j=i\theta_j$ is purely imaginary. 
The moduli space $\modm_{g,n}^{\mathrm{hyp}}(\mathbf{L})$ gives a connected component of the relative character variety of representations from $\pi_1(\Sigma)$ to $PSL(2,\br)$ with boundary holonomy determined by $\mathbf{L}$.

The moduli space $\modm_{g,n}^{\mathrm{hyp}}(\mathbf{0})$ is naturally homeomorphic to $\modm_{g,n}$, the moduli space of genus $g$ curves with $n$ labeled points, via the map that sends a complete hyperbolic surface to its conformal class and inverse map defined by uniformisation.  Here we use the complex analytic topology underlying the analytification of $\modm_{g,n}$, and similarly for its algebraic compactifications introduced later.  This generalises to a homeomorphism $\modm_{g,n}^{\mathrm{hyp}}(\mathbf{L})\cong\modm_{g,n}$ for $\mathbf{L}=(L_1,...,L_n)\in i[0,2\pi)^n$ satisfying a Gauss-Bonnet condition \eqref{areaineq}, via existence of conical metrics \cite{McOPoi}.  The homeomorphism $\modm_{g,n}^{\mathrm{hyp}}(\mathbf{L})\cong\modm_{g,n}$ is proven to also hold for $\mathbf{L}\in\br_{\geq 0}^n$, via a grafting procedure along any geodesic boundary component \cite{MonRie}.

On $\modm_{g,n}$, there is a naturally defined K\"ahler metric $g_{\mathrm{WP}}$, the Weil-Petersson metric \cite{WeiMod}, with K\"ahler form $\omega_{\mathrm{WP}}$, defined using the unique hyperbolic metric on a curve.  This defines the Weil-Petersson volume
\[\mathrm{Vol}(\modm_{g,n})=\int_{\modm_{g,n}}\exp\omega_{\mathrm{WP}}.\] 
Using the homeomorphism $\modm_{g,n}\cong\modm_{g,n}^{\mathrm{hyp}}(\mathbf{L})$ for $\mathbf{L}=(L_1,...,L_n)\in i[0,2\pi)^n$, the Weil-Petersson metric was generalised in \cite{STrWei} to define a K\"ahler metric  $\gpl$ for any $\mathbf{a}\in(0,1]^n$, and associated K\"ahler form  $\wpl$ over $\modm_{g,n}$.  The notation follows \cite{STrWei}, where one replaces $L\in[0,2\pi)$ by the weight $a(L)=1-L/2\pi i$ so that $\mathbf{a}=a(\mathbf{L})=(a_1,...,a_n)$ for $a_j=a(L_j)$ .  The family of K\"ahler forms $\wpl$ defines a deformation of the Weil-Petersson form $\omega_{\mathrm{WP}}=\omega_{(1,..,1)}$.  To emphasise the dependence on the conical hyperbolic metrics, the volume of $\modm_{g,n}$ with respect to $\wpl$ is denoted by $\mathrm{Vol}(\modm_{g,n}^{\mathrm{hyp}}(\mathbf{L})):=\int_{\modm_{g,n}}\exp\omega_{\mathbf{a}}$ for $\mathbf{a}=a(\mathbf{L})$.  

The symplectic form $\wpl$ can be defined in alternative ways.  Mondello \cite{MonPoi} proved that $\wpl$ agrees with the Goldman symplectic form \cite{GolInv} defined naturally on the character variety $\modm_{g,n}^{\mathrm{hyp}}(\mathbf{L})$ when $\mathbf{L}\in\{i[0,2\pi)\}^n$.  When $\mathbf{L}\in\{i[0,\pi)\}^n$, it is proven in \cite{AScGeo} that with respect to the Fenchel-Nielsen coordinates $\{\ell_j,\theta_j\mid j=1,...,3g-3+n\}$ for Teichm\"uller space associated to pants decompositions, $\wpl$ is defined by Wolpert's formula \cite{WolWei}:
\begin{equation}  \label{WPform}
\wpl=\sum d\ell_j\wedge d\theta_j.
\end{equation} 
Equivalently, $\{\ell_j,\theta_j\}$, $j=1,...,3g-3+n$, are Darboux coordinates for the symplectic form.  The formula \eqref{WPform} is mapping class group invariant hence descends to a well-defined symplectic form on the moduli space.  For more general values of $\mathbf{L}$, if there exist pants decompositions hence Fenchel-Nielsen coordinates, such as for $\mathbf{L}\in\br_{\geq 0}^n$, then  \eqref{WPform} is used to define a symplectic deformation $\omega_{\mathrm{WP}}(\mathbf{L})$ of the Weil-Petersson form.  

The Liouville volume form of $\omega_{\mathrm{WP}}$ and its deformations equips $\modm_{g,n}$ 
with a measure.  The measure can also be defined by torsion of the natural complex defined by the associated flat $PSL(2,\br)$ connection \cite{WitQua}.  See \cite[equation 3.44]{SWiJTG} for a proof that the measure constructed via torsion coincides with the measure induced from the symplectic form defined via Fenchel-Nielsen coordinates.  The total measure is finite and referred to as the volume of the moduli space.

When $\mathbf{L}\in\br_{\geq 0}^n$, Mirzakhani \cite{MirSim} proved that $\mathrm{Vol}(\modm_{g,n}^{\mathrm{hyp}}(\mathbf{L}))=\int_{\modm_{g,n}^{\mathrm{hyp}}(\mathbf{L})}\exp\omega_{\mathrm{WP}}(\mathbf{L})$ is finite, 
 and given by a symmetric polynomial $V_{g,n}(L_1,...,L_n)\in\br[L_1^2,...,L_n^2]$.  She proved this in two ways: 
\begin{enumerate}[label=(\Roman*)]
\item \label{recpol} via a recursion \cite{MirSim} which uniquely determines the volumes from the initial calculations\\ $\mathrm{Vol}(\modm_{0,3}^{\mathrm{hyp}}(L_1,L_2,L_3))=1$ and $\mathrm{Vol}(\modm_{1,1}^{\mathrm{hyp}}(L))=\frac{1}{48}(L^2+4\pi^2)$; and 
\item \label{compol}  via a symplectic reduction argument \cite{MirWei} combined with the result of Wolpert \cite{WolWei} that $\omega^{\mathrm WP}$ extends to a closed current on the Deligne-Mumford compactification $\overline{\modm}_{g,n}$, which represents the cohomology class $2\pi^2\kappa_1\in H^{2}( \overline{\cal M}_{g,n},\br)$.  When $\mathbf{L}\neq0$, Mirzakhani generalised this to show that $\omega_{\mathrm{WP}}(\mathbf{L})$ extends to define the cohomology class 
\begin{equation}  \label{WPdef}
\left[\omega_{\mathrm{WP}}(\mathbf{L})\right]=\left[\omega_{\mathrm{WP}}\right]+\sum_1^n\tfrac12L_j^2\psi_j\in H^2(\overline{\modm}_{g,n},\br)
\end{equation}
so that 
\begin{equation}  \label{Poldef}
\int_{\modm_{g,n}^{\mathrm{hyp}}(\mathbf{L})}\exp\omega_{\mathrm{WP}}(\mathbf{L})=\int_{\overline{\modm}_{g,n}}\exp\big(2\pi^2\kappa_1+\sum_1^n\tfrac12L_j^2\psi_j\big)\in\br[L_1,...,L_n].
\end{equation}
\end{enumerate}
The notation in \eqref{WPdef} identifies $\omega_{\mathrm{WP}}$ and $\omega_{\mathrm{WP}}(\mathbf{L})$ with their extensions as currents 
to $\overline{\modm}_{g,n}$.    

We {\em define} the polynomial $V_{g,n}(\mathbf{L})$ using the algebraic geometric formula \eqref{Poldef}
\[ V_{g,n}(\mathbf{L}):=\int_{\overline{\modm}_{g,n}}\exp\big(2\pi^2\kappa_1+\sum_1^n\tfrac12L_j^2\psi_j\big).
\]
It is natural to ask whether the evaluation of $V_{g,n}(\mathbf{L})$ at any $\mathbf{L}\in\left\{\br_{\geq0}\cup i\hspace{.5mm} \br_{>0}\right\}^n$ corresponds to the volume of the moduli space $\modm_{g,n}^{\mathrm{hyp}}(\mathbf{L})$.  Indeed, Mirzakhani's results show that the answer is yes for $\mathbf{L}\in\br_{\geq 0}^n$, and her recursion argument \ref{recpol} can be generalised to also give an affirmative answer when $\mathbf{L}\in\{\br_{\geq 0}\cup i[0,\pi)\}^n$ via a McShane identity on hyperbolic surfaces with cone angles, in which imaginary arguments appear \cite{TWZGen}.  The argument in \cite{TWZGen} relies on the existence of pants decompositions in any isotopy class together with Fenchel-Nielsen coordinates, and fails when any $\theta_j>\pi$, since pants decompositions are no longer guaranteed---see Appendix~\ref{nopop}.  

The answer is sometimes no, i.e.\ evaluation of a polynomial $V_{g,n}(\mathbf{L})$ at $L_j=i\theta$ for $\theta>\pi$ does not necessarily correspond to an actual volume. For example, Mirzakhani's polynomial $V_{0,4}(L_1,L_2,L_3,L_4)=2\pi^2+\frac12\sum_{j=1}^4L_j^2$ evaluates to be negative in the following example, 
\[V_{0,4}(0,0,\theta i,(2\pi-\epsilon)i)<0\quad\Leftarrow\quad 0<4\pi\epsilon<\theta^2<4\pi^2\] 
whereas the volume of the non-empty moduli space $\modm_{0,4}^{\mathrm{hyp}}(\mathbf{L})$ for $\mathbf{L}=(0,0,\theta i,(2\pi-\epsilon)i)$ is necessarily positive (and given explicitly by $\mathrm{Vol}(\modm_{0,4}^{\mathrm{hyp}}(\mathbf{L}))=\epsilon(2\pi-\theta)>0$).

Via relations between intersection numbers on $\overline{\modm}_{g,n}$, the polynomials $V_{g,n}(\mathbf{L})$ were proven in \cite{DNoWei} to satisfy the following relations:
\begin{align}  \label{limit}
\begin{split}
V_{g,n+1}({\bf L},2\pi i)&=\sum_{k=1}^n\int_0^{L_k}L_kV_{g,n}({\bf L})dL_k\\
\frac{\partial V_{g,n+1}}{\partial L_{n+1}}({\bf L},2\pi i)&=2\pi i(2g-2+n)V_{g,n}({\bf L}).
\end{split}
\end{align}
These relations were generalised in \cite{DuSim}.  The relations \eqref{limit} evoke the idea of a $2\pi $ limit of a cone angle, however until now no geometric meaning has been deduced from these relations.  This is because evaluation of the polynomials does not necessarily correspond to volumes close to the limit $\theta\to 2\pi$.    
One outcome of this paper is to describe occurrences when the relations \eqref{limit} have geometric meaning.   

The Weil-Petersson metric is incomplete on $\modm_{g,n}$, \cite{WolNon}, and its completion is compact and homeomorphic to  the Deligne-Mumford compactification $\overline{\modm}_{g,n}$, \cite{MasExt}.  For $\mathbf{L}\in i[0,2\pi)^n$, Schumacher and Trapani \cite{STrWei} generalised this to prove that the metric completion of $\modm_{g,n}$ with respect to $\gpl$ is also compact.  We denote the metric completion, which is a compact topological space, by $\overline{\modm}_{g,n}^{\mathrm{hyp}}(\mathbf{L})$. 

In general, $\overline{\modm}_{g,n}^{\mathrm{hyp}}(\mathbf{L})$ is not homeomorphic to the Deligne-Mumford compactification $\overline{\modm}_{g,n}$, which consists of all curves with at worst nodal singularities and finite automorphism group.  For $\mathbf{L}\in i[0,2\pi)^n$ it is instead homeomorphic to a compactification of $\modm_{g,n}$ via weighted curves, due to Hassett \cite{HasMod}.   Note that the metric completion is denoted by $\overline{\modm}_{g,\mathbf{a}}$ in \cite{STrWei}, but we reserve this notation for the Hassett compactification, which arises via an algebraic construction.   Schumacher and Trapani \cite{STrWei} proved that there is a natural continuous map  $\overline{\modm}_{g,n}^{\mathrm{hyp}}(\mathbf{L})\to\overline{\modm}_{g,n}^{\mathrm{hyp}}(\mathbf{L}')$ for $\mathbf{L}\leq\mathbf{L}'$, i.e. $|L_j|\leq|L_j'|$, $j=1,...,n$.  In particular, via the homeomorphism $\overline{\modm}_{g,n}\cong\overline{\modm}_{g,n}^{\mathrm{hyp}}(\mathbf{0})$ this induces a continuous map $\overline{\modm}_{g,n}\to\overline{\modm}_{g,n}^{\mathrm{hyp}}(\mathbf{L})$ that fits into the commutative diagram:
\begin{equation}  \label{homcom} 
\begin{array}{ccc}\modm_{g,n}&\stackrel{\cong}{\longrightarrow}&\modm_{g,n}^{\mathrm{hyp}}(\mathbf{L})\\
\downarrow&&\downarrow\\
\overline{\modm}_{g,n}&\longrightarrow
&\overline{\modm}_{g,n}^{\mathrm{hyp}}(\mathbf{L})
\end{array}
\end{equation}

The following theorem shows that the Weil-Petersson form $\wpl$, for $\mathbf{L}\in i[0,2\pi)^n$, extends to $\overline{\modm}_{g,n}$ as a closed current which defines a cohomology class expressible in terms of known tautological classes which agrees with \eqref{WPdef}.
\begin{theorem} \label{extension}
If
$\mathbf{L}\in\left\{(i\theta_1,...,i\theta_n)\in i[0,2\pi)^n\mid \theta_j+\theta_k<2\pi,\ \forall j\neq k  \right\}$ 
then $\wpl$ extends as a closed current  to $\overline{\modm}_{g,n}$ which defines the cohomology class
\begin{equation}  \label{WPcoh}
\left[\wpl\right]=2\pi^2\kappa_1-\tfrac12\sum_j\theta_j^2\psi_j\in H^2(\overline{\modm}_{g,n},\br).
\end{equation}
\end{theorem}
The proof of Theorem~\ref{extension} uses natural restriction properties of both sides of \eqref{WPcoh} and an induction argument to reduce the proof to calculations in a finite number of cases.  Mirzakhani's derivation \cite{MirWei} of the relation in cohomology \eqref{WPdef}, similar to  \eqref{WPcoh}, for the extension of the Weil-Petersson form $\wpl$ when $\mathbf{L}\in\br_{\geq 0}$ uses a different argument, given by a symplectic reduction argument together with a homeomorphism of a natural compactification of $\modm_{g,n}^{\mathrm{hyp}}(\mathbf{L})$ with the Deligne-Mumford compactification $\overline{\modm}_{g,n}$. 
Theorem~\ref{extension} leads to new examples of moduli spaces which have volume coinciding with evaluation of Mirzakhani's polynomial.
\begin{corollary}  \label{th:vol}
For 
$\mathbf{L}\in\left\{(i\theta_1,...,i\theta_n)\in i[0,2\pi)^n\mid \theta_j+\theta_k<2\pi,\ \forall j\neq k  \right\}$
\[\mathrm{Vol}(\modm_{g,n}^{\mathrm{hyp}}(\mathbf{L}))=V_{g,n}(\mathbf{L})=\int_{\overline{\modm}_{g,n}}\exp(2\pi^2\kappa_1+\tfrac12\mathbf{L}\cdot\mathbf{L})\]
i.e.\ the volume of the moduli space is obtained by evaluation of Mirzakhani's polynomial. 
\end{corollary}
The extension of $\wpl$ to $\overline{\modm}_{g,n}$ as a closed current used in Theorem~\ref{extension} holds more generally, for any cone angles $(\theta_1,...,\theta_n)\in [0,2\pi)^n$.  The existence of the extension is a consequence of the fact that points of the compactification $\overline{\modm}_{g,n}^{\mathrm{hyp}}(\mathbf{L})$ for any $\mathbf{L}\in i[0,2\pi)^n$ are represented by nodal hyperbolic surfaces.  
The condition $\theta_j+\theta_k<2\pi,\ \forall j\neq k$ in Theorem~\ref{extension} is a necessary and sufficient condition for the cohomology class of the extension $\left[\wpl\right]$ to take the form \eqref{WPcoh}.  In other words, if $\theta_j+\theta_k>2\pi$ for some $j\neq k$ then \eqref{WPcoh} no longer holds.  Necessity uses the following theorem which gives a homeomorphism of $\overline{\modm}_{g,n}^{\mathrm{hyp}}(\mathbf{L})$ with Hassett's compactification.  
The general form of the cohomology class $\left[\wpl\right]$ and its associated volume is pursued further in \cite{AMNWei}.

\begin{theorem}  \label{comphomeo}
Given $\mathbf{L}\in i[0,2\pi)^n$ which satisfies $\left|\sum_{j=1}^nL_j\right|<2\pi(2g-2+n)$, then there is a natural homeomorphism
\begin{equation}   \label{HassIso}
\overline{\modm}_{g,n}^{\mathrm{hyp}}(\mathbf{L})\cong\overline{\modm}_{g,\mathbf{a}}
\end{equation}  
where $\overline{\modm}_{g,\mathbf{a}}$ is the weighted Hassett space with $\mathbf{a}=a(\mathbf{L})$, i.e. $a_j=1-L_j/2\pi i$ for $j=1,...,n$. 
\end{theorem}

In the case of a single cone point with angle $\theta\in [0,2\pi)$, together with $n$ cusps, so that $\mathbf{L}=(0^n,i\theta)$, a consequence of Corollary~\ref{th:vol} is that Mirzakhani's polynomials give the volume of $\modm_{g,n}^{\mathrm{hyp}}(\mathbf{L})$ for all $\theta\in [0,2\pi)$. In this case one can consider the limit as the angle $\theta$ approaches $2\pi$, and thus attribute geometric meaning to the relations \eqref{limit}. 
\begin{corollary}  \label{old}
In the $2\pi$ cone angle limit, the volumes satisfy
\begin{align*}
\lim_{\theta\to 2\pi}\mathrm{Vol}(\modm_{g,n+1}^{\mathrm{hyp}}(0^n,i\theta))&=0\\
 \lim_{\theta\to 2\pi}\frac{\partial}{\partial \theta} \mathrm{Vol}(\modm_{g,n+1}^{\mathrm{hyp}}(0^n,i\theta))&=2\pi (2-2g-n)\mathrm{Vol}(\modm_{g,n}^{\mathrm{hyp}}(0^n)).
\end{align*}
\end{corollary}
The relations in Corollary~\ref{old} are a consequence of Corollary~\ref{th:vol} together with the intersection theory arguments of \cite{DNoWei}.   An alternative more direct geometric argument can be given using the degeneration of the K\"ahler metric $\gpl$ in the limit.  The next corollary is a positivity result.
\begin{corollary}
\begin{equation}  \label{positive}
\theta_j+\theta_k<2\pi,\quad \forall j\neq k\in\{1,\ldots,n\} \quad\Rightarrow\quad V_{g,n}(i\theta_1,...,i\theta_n)>0.
\end{equation}
\end{corollary}
The coefficients of Mirzakhani's polynomials are non-negative, hence positivity of the polynomials on $\mathbf{L}\in\br^n_{\geq 0}$ is obvious, even without the volume interpretation.  The inequality \eqref{positive} is not obvious as a property of the polynomials since 
the general case of imaginary $L_j$ is subtler, as shown in the following example. 
\begin{align*}
V_{0,5}(\mathbf{L})&=\frac18\left.\left(\sum_{j=1}^5L_j^4+4\sum_{j<k}^5L_j^2L_k^2+24\pi^2\sum_{j=1}^5L_j^2+80\pi^4\right)\right|_{L_j=i\theta_j}\\
&=\frac{1}{24}\sum
(4\pi^2-\theta_j^2-\theta_k^2-\theta_\ell^2)
(4\pi^2-\theta_j^2-\theta_m^2-\theta_n^2)>0\quad\Leftarrow\theta_j+\theta_k<2\pi, \forall j\neq k.
\end{align*}
Proof of positivity uses the expression for $V_{0,5}(\mathbf{L})$ in the second line above. 
The sum in the second line is over all $\{j,k,\ell,m,n\}=\{1,2,3,4,5\}$ and contains 15 summands.  In the case that $\theta_j+\theta_k<2\pi, \forall j\neq k$ one can deduce that $\theta_j^2+\theta_k^2+\theta_\ell^2<4\pi^2$, $\forall$ distinct $j,k,\ell$, so that
the second line implies positivity.  

We describe the natural compactifications $\overline{\modm}_{g,n}^{\mathrm{hyp}}(\mathbf{L})$ of $\modm_{g,n}$ in Section~\ref{sec:compact}.  Their homeomorphism type, given in Theorem~\ref{comphomeo}, is obtained by studying the universal curve over $\modm_{g,n}$ and its extension to the compactification, which consists of nodal hyperbolic surfaces as expected.  It is further shown that a compactification via the addition of nodal curves fails for some $\mathbf{L}$, in particular when the boundary components include both geodesics and cone angles greater than $\pi$.  In Section~\ref{WPsymp} we calculate the cohomology class of the extension of the generalised Weil-Petersson symplectic form $\wpl$ on $\modm_{g,n}$ to $\overline{\modm}_{g,n}$ analogous to Wolpert's result in the classical case.  This allows the calculation of the volume via intersection theory.   The calculation of $\mathrm{Vol}(\modm_{1,1}^{\mathrm{hyp}}(i\theta))$, which is an essential base case for the inductive proof of Theorem~\ref{extension}, is given in Appendix~\ref{vol11}.
 
\noindent {\em Acknowledgements.}   Research of PN was supported under the  Australian Research Council {\sl Discovery Projects} funding scheme project number DP180103891. Research of LA was supported by an Australian Government Research Training Program Scholarship and an Elizabeth \& Vernon Puzey Scholarship.

\section{Compactifications}   \label{sec:compact}
In this section we study the natural compactification of the moduli space $\modm_{g,n}$ obtained via the metric completion of the K\"ahler metric $\gpl$ for $\mathbf{a}=a(\mathbf{L})$ and any $\mathbf{L}\in \{i[0,2\pi)\}^n$.  We show that points in the compactification correspond to nodal hyperbolic surfaces with cusps at nodes.   The role of complete nodal hyperbolic surfaces in the compactification of the moduli space was first studied by Bers \cite{BerSpa} and developed further in \cite{HarCha,HKoAna}.   

\subsection{Metric completions}   \label{compl}
We begin with the definition of the Weil-Petersson metric $g_{\mathrm{WP}}$ over $\modm_{g,n}$ and its conical generalisation $\gpl$ where $\mathbf{a}=a(\mathbf{L})$ for $\mathbf{L}\in i[0,2\pi)^n$.  The tensor $g_{\mathrm{WP}}$ is naturally defined on the tangent space of $\modm_{g,n}$, although as usual it will be convenient to equivalently define it over the cotangent space.
The tangent space to $\modm_{g,n}$  at a point $(C,p_1,\dots,p_n)\in\modm_{g,n}$ is naturally given by $T_{[C]}\modm_{g,n}=H^1(C,T_C(-D))$ where $T_C(-D)$ is the sheaf of local holomorphic vector fields that vanish on $D=(p_1,\dots,p_n)$.   The cotangent space is naturally identified with the vector space of meromorphic quadratic differentials with simple poles on $D$:
\[T^*_{[C]}\modm_{g,n}=H^1(C,T_C(-D))^\vee\cong H^0(C,K^{\otimes 2}_C(D)).\]
At a point $(C,p_1,\dots,p_n)\in\modm_{g,n}$, the Weil-Petersson metric $g_{\mathrm{WP}}$ and Weil-Petersson form $\omega_{\mathrm{WP}}$ are defined as the real and imaginary part of the following Petersson pairing on the vector space of quadratic differentials:
\begin{equation}  \label{WPmetric}
\langle\eta,\xi\rangle:=\int_C\frac{\overline{\eta}\xi}{h},\qquad \eta,\xi\in H^0(C,K_C^{\otimes 2}(D))
\end{equation}
where $h$ is the complete hyperbolic metric on $C-D$.  

When $\mathbf{L}\in i[0,2\pi)^n$ and 
\begin{equation}  \label{areaineq}
\sum_{j=1}^n|L_j|<2\pi(2g-2+n)
\end{equation}
there is a homeomorphism $\modm_{g,n}^{\mathrm{hyp}}(\mathbf{L})\cong\modm_{g,n}$ which associates to any conical hyperbolic metric its conformal class, and conversely the existence of a unique conical metric $h(\mathbf{L})$ with prescribed cone angles $|L_j|$, $j=1,...,n$ in the given conformal class is proven by McOwen in \cite{McOPoi}.  The inequality \eqref{areaineq} is a necessary condition for the moduli space $\modm_{g,n}^{\mathrm{hyp}}(\mathbf{L})$ to be non-empty due to the Gauss-Bonnet theorem.    Schumacher and Trapani \cite{STrWei} generalised \eqref{WPmetric} by replacing the complete hyperbolic metric $h$ with a conical hyperbolic metric $h(\mathbf{L})$ with prescribed cone angles at points of $D$.   In terms of the cotangent bundle it is defined as follows.
\begin{definition}  \label{wpa}
For $\mathbf{a}=a(\mathbf{L})$, where $\mathbf{L}\in i[0,2\pi)^n$ satisfies \eqref{areaineq}, define 
$\gpl$, respectively 
$\wpl$, as the real, respectively imaginary, part of the Hermitian metric: %
\begin{equation}  \label{WPcmetric}
\langle\eta,\xi\rangle:=\int_C\frac{\overline{\eta}\xi}{h(\mathbf{L})},\qquad \eta,\xi\in H^0(C,K_C^{\otimes 2}(D)).
\end{equation}
\end{definition}
It is proven in \cite{STrWei} that $\gpl$ is  a K\"ahler metric.   The metric completion of $\modm_{g,n}$ with respect to $\gpl$ defines $\overline{\modm}_{g,n}^{\mathrm{hyp}}(\mathbf{L})$ .
\begin{remark}
For general $L\in i\br_+$, where one allows cone angles of $2\pi$ and greater, it is proven in \cite{McOPoi} that existence of a hyperbolic metric with given cone angles in a given conformal class still holds.  However, \eqref{WPcmetric} diverges when the angle is not less than $2\pi$ so it no longer defines a metric.
\end{remark}

Theorem~\ref{comphomeo} is a consequence of the following four propositions.

\begin{proposition}[\cite{STrWei}]   \label{modcomp}
The identity map on $\modm_{g,n}$ extends to a continuous, surjective map
\begin{equation}  \label{surj}
\overline{\modm}_{g,n}\to\overline{\modm}_{g,n}^{\mathrm{hyp}}(\mathbf{L})
\end{equation}
and in particular $\overline{\modm}_{g,n}^{\mathrm{hyp}}(\mathbf{L})$ is compact.
\end{proposition}
\begin{proof}
It is proven in \cite[Proposition 2.4]{STrVar} that if $|\mathbf{L}|\geq |\mathbf{L}'|$, i.e. $|L_j|\geq |L_j'|$ for each $j$, then on a given curve $C$ the hyperbolic metric $h(\mathbf{L}')$ obtained via \cite{McOPoi}  dominates $h(\mathbf{L})$ i.e. they satisfy $h(\mathbf{L})\leq h(\mathbf{L}')$.   The formula \eqref{WPcmetric} then implies that the corresponding  K\"ahler metrics satisfy the following monotonicity result:
\begin{equation}   \label{metcomp}
\mathbf{a}\leq \mathbf{a}'\quad\Rightarrow\quad\gpl\leq g_{\mathbf{a}'}.
\end{equation}
  An immediate consequence is that there exists a distance decreasing, hence continuous, map between the metric completions
\[\overline{\modm}_{g,n}^{\mathrm{hyp}}(\mathbf{L}')\to\overline{\modm}_{g,n}^{\mathrm{hyp}}(\mathbf{L}).
\] 
Applying this to $\mathbf{L}'=\mathbf{0}$ produces a continuous map
\[
\overline{\modm}_{g,n}\to\overline{\modm}_{g,n}^{\mathrm{hyp}}(\mathbf{L})
\]
since the completion $\overline{\modm}_{g,n}^{\mathrm{hyp}}(\mathbf{0})$ is homeomorphic to the Deligne-Mumford compactification $\overline{\modm}_{g,n}$, \cite{MasExt}.  Its image is compact, since it is the image of a compact space under a continuous map, and dense since it is the identity on $\modm_{g,n}$ which is dense in its metric completion.  The image is closed and dense, hence all of $\overline{\modm}_{g,n}^{\mathrm{hyp}}(\mathbf{L})$ which is necessarily compact.
\end{proof}
Proposition~\ref{modcomp} produces compactifications of $\modm_{g,n}$ without knowledge of the moduli functor, i.e.\ an understanding of  the universal curve over $\modm_{g,n}$ and whether it extends over each compactification with geometrically meaningful fibres.  Indeed, by studying the behaviour of the universal curve over $\modm_{g,n}$ as a point approaches the boundary in $\overline{\modm}_{g,n}^{\mathrm{hyp}}(\mathbf{L})$ one sees that nodal hyperbolic surfaces arise, and this can be used to determine the homeomorphism type of the metric completion stated in Theorem~\ref{comphomeo}.

A complete {\em nodal hyperbolic surface} $\Sigma$ is a nodal surface equipped with a complete hyperbolic metric on each component of $\Sigma-N$, where $N\subset\Sigma$ denotes the nodes, or double points, of $\Sigma$.  More generally, a nodal hyperbolic surface with cone points  labeled by $\mathbf{L}\in i\hspace{.5mm} [0,2\pi)^n$ (or more generally $\mathbf{L}\in\{\br_{\geq0}\cup i\hspace{.5mm} \br_{>0}\}^n$) is a nodal surface equipped with a hyperbolic metric on each component $\Sigma'\subset\Sigma-N$ which is complete at nodes.  Each irreducible component $\Sigma'\subset\Sigma-N$ is a hyperbolic surface with cone points and geodesic boundary components determined by those $L_i$ corresponding to each $p_i\in\Sigma'$, and each node is a hyperbolic cusp.

Given any point in $\overline{\modm}_{g,n}$ there exists a unique complete nodal hyperbolic surface with irreducible components conformally equivalent to the corresponding irreducible components of the stable curve.  Moreover, any complete nodal hyperbolic surface lies in the metric completion $\overline{\modm}_{g,n}^{\mathrm{hyp}}(\mathbf{0})$ of $\modm_{g,n}$ with respect to the Weil-Petersson metric, due to the homeomorphism $\overline{\modm}_{g,n}^{\mathrm{hyp}}(\mathbf{0})\cong\overline{\modm}_{g,n}$ \cite{MasExt}.  The following proposition generalises this to any nodal hyperbolic surface with cone angles.

\begin{proposition}   \label{path}
Given a nodal hyperbolic surface $\Sigma$ with node at $P\in\Sigma$ and cone points labeled by $\mathbf{L}\in i\hspace{.5mm} [0,2\pi)^n$ there exists a finite length path $\gamma:[0,1)\to\modm_{g,n}$, with respect to the metric $\gpl$ for $\mathbf{a}=a(\mathbf{L})$, such that the associated path of smooth hyperbolic surfaces $\{\Sigma_t\}_{t\in[0,1)}$ converges to the nodal surface $\Sigma$.  In particular, $\Sigma$ represents a point in $\overline{\modm}_{g,n}^{\mathrm{hyp}}(\mathbf{L})$.
\end{proposition}
\begin{proof}
The idea of the proof is to show that the cusps on either side of $P$ can be deformed to short geodesics, hence smoothed away.  Then a comparison result will show that the path of smooth hyperbolic surfaces, which converges to the nodal surface, has finite length.  We first study the deformation of a cusp on a smooth hyperbolic surface.

The moduli space $\modm_{g,n}^{\mathrm{hyp}}(L_1,...,L_n)$ maps as an open set to the character variety of conjugacy classes of representations $\pi_1(C)\to PSL(2,\br)$.  The character variety is isomorphic to $PSL(2,\br)^M/PSL(2,\br)$ where $M=2g-1+n$ is the rank of the free group $\pi_1(C)$.  For a simple closed curve $\beta_j\subset C$ representing a boundary component (in the case that $n>0$), the function $|\mathrm{tr}(\beta_j)|$ is regular on $\modm_{g,n}^{\mathrm{hyp}}(L_1,...,L_n)$ at the value 2 corresponding to $L_j=0$.  This follows from regularity of $|\mathrm{tr}(A)|$ on $PSL(2,\br)_*=PSL(2,\br)\backslash\{\pm I\}$ at the value 2, i.e.\hspace{1mm}surjectivity of $D|\mathrm{tr}(\beta_j)|$ at each point of $|\mathrm{tr}(\beta_j)|^{-1}(2)$, which follows from direct calculation.  Hence $|\mathrm{tr}(A_j)|$ is regular at the value 2 on $PSL(2,\br)_*^M$ where $A_j$ is the $j$th factor,  and it is also regular at the value 2 on $PSL(2,\br)_*^M/PSL(2,\br)$ since $|\mathrm{tr}(A_j)|$ and its linearisation are invariant under conjugation.
    By the regularity of $|\mathrm{tr}(A_j)|$, a cusp, or any number of cusps, can be deformed to small length boundary geodesics by varying $|\mathrm{tr}(A_j)|=2$ to $|\mathrm{tr}(A_j)|>2$, which produces a homeomorphic moduli space via an application of Ehresmann’s theorem.  Hence, given a hyperbolic surface with cusps, there exists a path of smooth hyperbolic surfaces with boundary geodesics tending to the cusps, i.e. the boundary geodesic lengths tend to zero, so that the path ends at the hyperbolic surface with cusps.  

It follows that given a nodal hyperbolic surface with node at $P$, there exists a path of smooth hyperbolic surfaces with boundary geodesics tending to the cusps at $P$, so that the path ends at the nodal hyperbolic surface.  This is achieved by gluing a path of hyperbolic surfaces along two equal length geodesic boundaries that tend to each side of the cusp $P$ via the construction above.  

The path of smooth hyperbolic surfaces has finite length with respect to the metric $\gpl$, since by \eqref{metcomp} the length is less than the length of the path with respect to the Weil-Petersson metric $g_{\mathrm{WP}}$, using the corresponding path of {\em complete} hyperbolic surfaces, which is finite by \cite{MasExt}.  Hence the path converges to the given nodal hyperbolic surface inside the metric completion $\overline{\modm}_{g,n}^{\mathrm{hyp}}(\mathbf{L})$ as required.   Note that the ambiguity in gluing due to rotations along the geodesic shows that the path may be naturally considered as a family over a disk.  The argument immediately generalises to a collection of nodes.  
\end{proof} 
 
\begin{proposition}   \label{mainch}
The map \eqref{surj} defines a homeomorphism
when $|L_j|+|L_k|<2\pi,\ \forall j\neq k$.
\end{proposition}
\begin{proof} We must show that the map \eqref{surj} is injective, since an injective continuous map from a compact domain is necessarily a homeomorphism onto its image.  

Given any nodal curve $C$ represented by a point in $\overline{\modm}_{g,n}$, on each irreducible component $Y\subset C$ there exists a unique hyperbolic metric with cone angle at each $p_j\in Y$ given by $|L_j|$, and cusps at nodes.  The existence uses the assumption $|L_j|+|L_k|<2\pi,\ \forall j\neq k$ which guarantees that each irreducible component, in particular a rational component with one node, satisfies \eqref{areaineq}.  The union over all irreducible components of these smooth hyperbolic surfaces forms a nodal hyperbolic surface which, by Proposition~\ref{path}, represents a point in the metric completion $\overline{\modm}_{g,n}^{\mathrm{hyp}}(\mathbf{L})$.  Any other nodal curve $C'$ gives rise to a different nodal hyperbolic surface which represents a point in the metric completion $\overline{\modm}_{g,n}^{\mathrm{hyp}}(\mathbf{L})$.  This defines an embedding of $\overline{\modm}_{g,n}$ into $\overline{\modm}_{g,n}^{\mathrm{hyp}}(\mathbf{L})$.

It remains to show that every point in the metric completion $\overline{\modm}_{g,n}^{\mathrm{hyp}}(\mathbf{L})$ corresponds to a nodal hyperbolic surface which represents a stable curve in $\overline{\modm}_{g,n}$.
Consider a path $\gamma:(0,1]\to\overline{\modm}_{g,n}$ with limit point $\gamma(1)$ in the boundary, so that $\gamma(1)$ defines a nodal curve $C$. 
In each conical hyperbolic surface corresponding to $\gamma(t)$ for $t<1$, there exists a unique simple closed geodesic in the isotopy class of the vanishing cycle, with length that tends to zero as $t\to 1$.  This is proven via comparison with the complete case using the homeomorphism $\modm_{g,n}^{\mathrm{hyp}}(\mathbf{L})\cong\modm_{g,n}\cong\modm_{g,n}^{\mathrm{hyp}}(\mathbf{0})$ and the following argument.  The path $\gamma(t)$ also produces a path of {\em complete} hyperbolic surfaces containing simple closed geodesics, say of length $\epsilon(t)$, that become a cusp in the limit so $\lim_{t\to1}\epsilon(t)=0$.  This is a consequence of Mumford's compactness criterion \cite{MumRem}, that given $\epsilon>0$, the subset of $\modm_{g,n}$ corresponding to hyperbolic surfaces containing no closed geodesic of length less than $\epsilon$ is compact in $\modm_{g,n}$. 
The simple closed geodesic on the complete hyperbolic surface produces a simple closed curve on the conical hyperbolic surface that is no longer a geodesic but has length $\leq\epsilon(t)$ by the monotonicity result $h(\mathbf{L})\leq h(\mathbf{0})$ from \cite{STrVar}, where $h(\mathbf{L})$ is the conical metric and $h(\mathbf{0})$ is the complete metric.  The simple closed geodesic on the conical hyperbolic surface is in the same isotopy class and has length $\leq\epsilon(t)$.  To summarise, the existence of simple closed geodesics of lengths $\epsilon(t)\to0$ on the complete hyperbolic surfaces gives rise to simple closed geodesics on the conical hyperbolic surface with lengths tending to zero, as claimed. 

Each component of the complement of the simple closed geodesics with length tending to zero is a hyperbolic surface of constant positive area, by the Gauss-Bonnet theorem.  Hence the path of conical hyperbolic surfaces tends to a nodal conical hyperbolic surface in the limit.  Furthermore, the distance between cone points is bounded below by Lemma~\ref{angle condition repulsion}, hence the cone points remain a positive distance apart  in the limit.  So the limiting nodal hyperbolic surface $\Sigma$ has cone points labeled by $\mathbf{L}\in i\hspace{.5mm} [0,2\pi)^n$, and each irreducible component with cusps at nodes  satisfies the Gauss-Bonnet condition \eqref{areaineq} hence $\Sigma$ represents a nodal curve in $\overline{\modm}_{g,n}$.
\end{proof}

In general, the map \eqref{surj} is not a homeomorphism.  Here we introduce Hassett's compactifications $\overline{\modm}_{g,\mathbf{a}}$ for $\mathbf{a}\in(0,1]^n$ to help to understand \eqref{surj} in general.  Consider nodal curves with labeled points $(p_1,...,p_n)$, each weighted respectively by $\mathbf{a}=(a_1,...,a_n)$, and define the following stability condition.   A smooth curve is $\mathbf{a}$-stable if the weighted Euler characteristic $2-2g-\sum_{j=1}^na_j$ is negative.  A nodal curve $C$ is $\mathbf{a}$-stable if each irreducible component $C'\subset C$ is $\mathbf{a}'$-stable with respect to the weight $\mathbf{a}'$ given by the restriction of $\mathbf{a}$ to $C'$, where nodes are assigned weight 1.  
Hassett \cite{HasMod} proved that the space $\overline{\modm}_{g,\mathbf{a}}$ of $\mathbf{a}$-stable nodal curves is compact.  Geometrically, the compactification parametrises families of stable curves such that a set of points $\{p_j\mid j\in J\subset\{1,...,n\}\}$ for $|J|>1$ can coincide only when $\sum_{j\in J}a_j<1$.  One consequence is that labeled points cannot coincide with nodes. The Deligne-Mumford compactification corresponds to $a_j=1$ for $j=1,...,n$ (and more generally when $a_j+a_k>1$ for all $j,k$) in which case labeled points can never coincide.


Stable curves admit complete nodal hyperbolic metrics, i.e. complete hyperbolic metrics on each irreducible component, in the given conformal class.  Similarly,
when $\mathbf{a}=a(\mathbf{L})$, Hassett's stability condition is equivalent to the existence of nodal hyperbolic metrics with cone angles $|\mathbf{L}|$ in the given conformal class since the stability condition on each irreducible component agrees with the Gauss-Bonnet condition \eqref{areaineq}:
\[  \sum_{j\in S}|L_j|<2\pi(2g'-2+|S|)\ \Leftrightarrow\ 2-2g'-\sum_{j\in S}a_j<0.
\]

Since Deligne-Mumford stability implies $\mathbf{a}$-stability, there is a natural continuous map: 
\begin{equation}  \label{DMtoHas} 
\overline{\modm}_{g,n}\to\overline{\modm}_{g,\mathbf{a}}.
\end{equation}
For  $S\subset\{1,...,n\}$ and $|S|\geq 2$, define the divisor $D_S\subset\overline{\modm}_{g,n}$ to consist of all stable curves that contain a rational irreducible component with one node and the marked points $p_S$.   For each $S$ satisfying $\sum_{j\in S}a_j<1$, the map \eqref{DMtoHas} sends $D_S$ to a point, and is one-to-one on the dense subset given by the complement of all such $D_S$. 
\begin{proposition}   \label{factorHas}
The map \eqref{surj} factors through a homeomorphism with the Hassett space
\[\overline{\modm}_{g,n}\to\overline{\modm}_{g,\mathbf{a}}\stackrel{\cong}{\to}\overline{\modm}_{g,n}^{\mathrm{hyp}}(\mathbf{L})
\]
for $\mathbf{a}=a(\mathbf{L})$.
\end{proposition}
\begin{proof}
The statement of the proposition is proven in Proposition~\ref{mainch} for the case that $|L_j|+|L_k|<2\pi,\ \forall j\neq k$ which is equivalent to $a_j+a_k>1,\ \forall j\neq k$, since  $\overline{\modm}_{g,n}\cong\overline{\modm}_{g,\mathbf{a}}$ by \cite{HasMod}.
So we may assume that for some non-singleton $S\subset\{1,...,n\}$, $\sum_{j\in S}a_j<1$.  The argument uses a comparison with the complete case via the homeomorphism $\modm_{g,n}^{\mathrm{hyp}}(\mathbf{L})\cong\modm_{g,n}\cong\modm_{g,n}^{\mathrm{hyp}}(\mathbf{0})$, as in the proof of Proposition~\ref{mainch}, although the conclusion of the comparison differs, leading to collapsed components.  Consider a path $\gamma:(0,1]\to\overline{\modm}_{g,n}$ such that $\gamma:(0,1)\to\modm_{g,n}$ and $\gamma(1)\in D_S$, i.e. $\gamma(1)$ defines a nodal curve with a rational irreducible component with marked points $p_S$ and exactly one node. 

In each complete hyperbolic surface corresponding to $\gamma(t)$ for $t<1$, there exists a unique simple closed geodesic bounding a disk with cusps at $p_S$.   The path of simple closed geodesics becomes a cusp in the limit, so the geodesic has length $\epsilon(t)$ where $\lim_{t\to 1}\epsilon(t)=0$.  

When the surface is equipped with the conical hyperbolic metric with cone angles $|\mathbf{L}|$, a simple closed geodesic surrounding the cone points $p_S$ no longer exists since the Gauss-Bonnet formula would produce contradictory negative area:
\[\text{area}(C)=-2\pi+\sum_{j\in S}(2\pi-|L_j|)=2\pi(-1+\sum_{j\in S}a_j)<0.\]
Nevertheless, the simple closed geodesic on the complete hyperbolic surface produces a simple closed curve on the conical hyperbolic surface that is no longer a geodesic but has length $\leq\epsilon(t)$ by the monotonicity result $h(\mathbf{L})\leq h(\mathbf{0})$ from \cite{STrVar}, where $h(\mathbf{L})$ is the conical metric and $h(\mathbf{0})$ is the complete metric.  To summarise, the existence of a small length simple closed geodesic on the complete hyperbolic surface gives rise to a small length simple closed curve on the conical hyperbolic surface.  

Let $L\leq\epsilon(t)$ be the length of the simple closed curve which encloses a hyperbolic disk of area $A$ with cone points of angles $\theta_j=|L_j|$.
The disk satisfies an isoperimetric inequality due to A.D. Alexandrov, see for example \cite{IzmSim}:
\[ L^2 \geq 2(2\pi-\sum_{j\in S}(2\pi-\theta_j))A+A^2\geq A^2\]
where the second inequality uses the assumption $2\pi-\sum_{j\in S}(2\pi-\theta_j)\geq 0$.  This gives an area bound $A\leq\epsilon(t)\to0$ as $t\to 1$.

An area bound on the hyperbolic disk with cone points implies a diameter bound by the following standard argument.  Given one of the cone points $p_i$, consider an embedded radius $r$ disk neighbourhood that does not meet any other cone point nor the simple closed curve boundary.  Its area is bounded by the total area $A$, so 
$2\theta_i\sinh^2(r/2)<A\leq\epsilon(t)$ hence the supremum over all such $r$ is bounded by $C\epsilon(t)^{1/2}$.  Summing over all cone points gives a diameter bound $C'\epsilon(t)^{1/2}$.  The diameter bound shows that the disk containing the cone points contracts to a single point in the limit.  

The limit point of the contracted disk is a cone point of angle $\theta=\sum_{j\in S}(2\pi-|L_j|)$, due to Lemma~\ref{angles merging}, on the hyperbolic surface corresponding to the endpoint $\gamma(1)$.   This applies to all points of $D_S\subset\overline{\modm}_{g,n}$, and  so $D_S$ maps to a single point in $\overline{\modm}_{g,n}^{\mathrm{hyp}}(\mathbf{L})$ under the map \eqref{surj}.  Since the map $\overline{\modm}_{g,n}\to\overline{\modm}_{g,\mathbf{a}}$ contracts $D_S$ to a point, and this holds for all non-singleton $S\subset\{1,...,n\}$ such that $\sum_{j\in S}a_j<1$, we have proven that the map \eqref{surj} factors through the Hassett space.

It remains to show that the induced map $\overline{\modm}_{g,\mathbf{a}}\to\overline{\modm}_{g,n}^{\mathrm{hyp}}(\mathbf{L})$ is a homeomorphism.  The proof is similar to the proof of Proposition~\ref{mainch}.  It is enough to show that the induced map is injective, since it is surjective and the domain is compact.  Given any nodal curve $C$ represented by a point in $\overline{\modm}_{g,\mathbf{a}}$, there are two cases.   In the first case cone points merge to produce a new cone point, and in the second case cusps at a node appear (and both cases may occur together).  In both cases, this arises from a finite length path with respect to the metric $\gpl$ using the comparison \eqref{metcomp} with the Weil-Petersson metric.  In the first case, if cone points merge, then as in the proof of Proposition~\ref{path} any path in $\modm_{g,n}$ that converges to this curve point has finite length with respect to the metric $\gpl$, since the length is less than the length of the path with respect to $g_{\mathrm{WP}}$ which is finite by \cite{MasExt}. Hence the path converges to the given nodal hyperbolic surface inside the metric completion $\overline{\modm}_{g,n}^{\mathrm{hyp}}(\mathbf{L})$ as required. In the second case, when cusps at a node appear, on each irreducible component $Y\subset C$ there exists a unique hyperbolic metric with cone angle at each $p_j\in Y$ given by $|L_j|$, and cusps at nodes.  Existence of a conical hyperbolic metric is guaranteed by \cite{McOPoi} applied to $Y$ since it satisfies the Gauss-Bonnet condition \eqref{areaineq} or equivalently $\sum_{j\in S}a_j>1$.  This follows automatically if it has positive genus or more than one node, and by the assumption that for any rational irreducible component with marked points $p_S$ and exactly one node then $\sum_{j\in S}a_j>1$.  The union over all irreducible components of these smooth hyperbolic surfaces forms a nodal hyperbolic surface which, by Proposition~\ref{path} occurs along a finite length path of surfaces converging to a nodal surface, represents a point in the metric completion $\overline{\modm}_{g,n}^{\mathrm{hyp}}(\mathbf{L})$.  This produces a unique point in $\overline{\modm}_{g,n}^{\mathrm{hyp}}(\mathbf{L})$ for every nodal curve in $\overline{\modm}_{g,n}$, showing that the induced map is injective and that the induced map $\overline{\modm}_{g,\mathbf{a}}\to\overline{\modm}_{g,n}^{\mathrm{hyp}}(\mathbf{L})$ is a homeomorphism. 
\end{proof}

\subsection{Local behaviour near boundary components} 
Here we further study the behaviour of the universal curve over the metric completion $\overline{\modm}_{g,n}^{\mathrm{hyp}}(\mathbf{L})$ by characterising possible interactions between boundary components in a family of hyperbolic surfaces.   Lemmas~\ref{angle condition repulsion} and \ref{angles merging}, which give the behaviour of cone angles that may merge together, are required in the proofs of Propositions~\ref{mainch} and \ref{factorHas}.  
 Various cases involving geodesic boundary components and cone angles are given in Lemmas~\ref{length and small angle repulsion} - \ref{angle length merge}, and in particular a consequence of Lemma~\ref{angle length merge} shows the failure of nodal surfaces to produce a compactification in the general case. 

The first case, involving small cone angles and geodesic boundary lengths, generalises the positive lower bound on the distance between two cuspidal boundary components, calculated as the distance between horocycles of fixed radius, over all hyperbolic surfaces---it is $\log(4)$ for radius $1/2$ horocycles.  As the horocycles approach each other, a cusp forms elsewhere on the hyperbolic surface.  More generally, several cuspidal boundary components may approach each other, and $\log(4)$ remains a lower bound for the distance  between horocycles of radius $1/2$. Lemma~\ref{length and small angle repulsion} is a well-known consequence of the proof of Mumford's compactness criterion in \cite{FMaPri}---nevertheless we give here a proof with techniques common to the subsequent cases.
\begin{lemma}[\cite{BusCol,DPaCol,FMaPri,KeeCol}] \label{length and small angle repulsion}
Given $L_j\in\br_{> 0}\cup i[0,\pi)$, for any surface in $\cM_{g,n}^{\mathrm{hyp}}(L_1,\ldots,L_n)$ there exists a positive lower bound on distances between boundary components.
\end{lemma}
\begin{proof}
Let $L_1,\ldots,L_n\in\br_{> 0}\cup i[0,\pi)$ correspond to boundary components $\beta_1,\ldots,\beta_n$, respectively. Given an element of Teichm\"uller space $X\in\cT_{g,n}^{\mathrm{hyp}}(L_1,\ldots,L_n)$, define $P_{jk}(X)$ as the set of simple paths connecting $\beta_j$ and $\beta_k$.  Consider the functions on Teichm\"uller space \begin{align*}
    f_{jk}:\cT_{g,n}^{\mathrm{hyp}}(L_1,\ldots,L_n) &\longrightarrow \bbR_{\geq 0} \\
    X &\longmapsto \min_{\gamma\in P_{jk}(X)}{\ell(\gamma)},
\end{align*}
where $\min$ is taken over lengths of paths. The proof of the Lemma amounts to showing that $f_{jk}$ is bounded below by some positive number. \\ \indent
Let $X\in\cT_{g,n}^{\mathrm{hyp}}(L_1,\ldots,L_n)$, let $\gamma$ be any path from $L_1$ to $L_2$, and let $K$ be a closed tubular neighbourhood of $\beta_1\cap\beta_2\cap \gamma\subset X$. Its boundary $\partial K$ is a simple closed loop, and since $L_1,L_2\in\br_{\geq 0}\cup i[0,\pi)$ we can shorten $\partial K$ to produce a simple, closed geodesic $c\subset X$. This produces a hyperbolic pair-of-pants with boundary components $\beta_1$, $\beta_2$ and $c$. In the pair-of-pants, let $\gamma'$ denote the shortest curve in the isotopy class of $\gamma$, and suppose that $\ell(\gamma')=\delta$. We show that $\delta$ is bounded below by some positive number. There are three possibilities for the boundary components $L_1$ and $L_2$, which require separate treatment. Firstly, both $\beta_1$ and $\beta_2$ could be geodesic boundary components or cusps. Secondly, one of $\beta_1$ or $\beta_2$ could be geodesic boundary component or cusp, and the other a cone point. Thirdly, both $\beta_1$ and $\beta_2$ could be cone points. The first case is proven in \cite{BusCol} and \cite{KeeCol}, and the third case is proven in Dryden-Parlier \cite{DPaCol}.  We include proofs here for comparison with the case of large cone angle behaviour.\\
\textit{\underline{Case 1. $L_1,L_2\in\bbR_{\geq 0}$:}}\\ \indent
Cut along $\gamma'$ and the other two seams of the hyperbolic pair-of-pants, to obtain two right angle hyperbolic hexagons. One such hexagon is shown in Figure \ref{Hyperbolic hexagon}. 

\begin{figure}[!htb] \centering
\begin{tikzpicture}[scale=3]
    \tkzDefPoint(0,0){O}
    \tkzDefPoint(1,0){Z}

    \tkzDefPoint(-0.6,-0.1){A}
    \tkzDefPoint(-0.25,0.5){P}
    \tkzDefPoint(0.4,0.4){Q}
    \tkzDefPoint(0.4,-0.3){C}
    \tkzDefPoint(0.1,0.5){B}
    \tkzDefPoint(0.1,-0.6){D}
    \tkzDefPoint(-0.4,-0.5){E}

    \tkzDrawCircle[fill=white](O,Z)

    \tkzClipCircle(O,Z)


    \tkzDrawCircle(O,Z)

  \tkzDefCircle[orthogonal through=A and P](O,Z) \tkzGetPoint{L1}
\tkzDrawCircle[color=red](L1,A)

\tkzDefCircle[orthogonal through=P and Q](O,Z) \tkzGetPoint{L2}
\tkzDrawCircle[color=blue](L2,P)

\tkzDefCircle[orthogonal through=Q and C](O,Z) \tkzGetPoint{L3}
\tkzDrawCircle[color=blue](L3,Q)
   
\tkzDefCircle[orthogonal through=A and E](O,Z) \tkzGetPoint{L4}
\tkzDrawCircle[color=blue](L4,A)

\tkzDefCircle[orthogonal through=E and D](O,Z) \tkzGetPoint{L5}
\tkzDrawCircle[color=blue](L5,E)

\tkzDefCircle[orthogonal through=D and C](O,Z) \tkzGetPoint{L6}
\tkzDrawCircle[color=blue](L6,D)

    \node at (-0.5,0.2) {$\delta$} ;
    \node at (0.1,0.5) {$\frac{L_2}{2}$} ;
    \node at (-0.55,-0.3) {$\frac{L_1}{2}$} ;
    \node at (0.3,-0.5) {$\frac{c}{2}$} ;
\end{tikzpicture}
\caption{Right angled hyperbolic hexagon.} \label{Hyperbolic hexagon}
\end{figure}
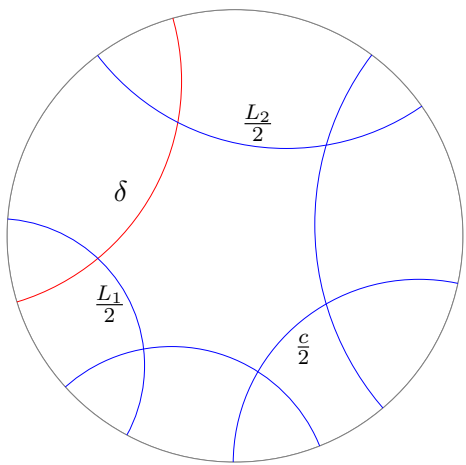  
 The hyperbolic hexagon formula $\cosh(c)=\sinh(a)\sinh(b)\cosh(\delta)-\cosh(a)\cosh(b)$ \cite[Theorem 2.4.1, (i)]{BusGeo} implies
    \begin{align}
        \cosh(\delta)&=\frac{\cosh(\frac{c}{2})+\cosh(\frac{L_1}{2})\cosh(\frac{L_2}{2})}{\sinh(\frac{L_1}{2})\sinh(b)} \geq\frac{1+\cosh(\frac{L_1}{2})\cosh(\frac{L_2}{2})}{\sinh(\frac{L_1}{2})\sinh(\frac{L_2}{2})} \nonumber \\
                     &>\frac{1}{\sinh(\frac{L_1}{2})\sinh(\frac{L_2}{2})} +1. \label{ineq1}
    \end{align}

    This means that $\delta>\varepsilon(L_1,L_2)$, for some positive function $\varepsilon$ of the lengths of the boundary geodesics. \\
\textit{\underline{Case 2. $L_1\in\bbR_{> 0}$ and $L_2\in i[0,\pi)$:}}\\ \indent
Let $\theta_2=|L_2|	$ be the cone angle defined by $L_2$. Cut along $\gamma'$ and the other two seams of the hyperbolic pair-of-pants, to obtain two hyperbolic pentagons, each with $4$ right angles and one angle $\frac{\theta_2}{2}$ at the cone point. One such pentagon is shown in Figure \ref{Hyperbolic pentagon}.
\begin{figure}[!htb] \centering
\begin{tikzpicture}[scale=3]
    \tkzDefPoint(0,0){O}
    \tkzDefPoint(1,0){Z}

    \tkzDefPoint(-0.6,-0.1){A}
    \tkzDefPoint(0.4,-0.3){C}
    \tkzDefPoint(0.1,0.5){p}
    \tkzDefPoint(0.1,-0.6){D}
    \tkzDefPoint(-0.4,-0.5){E}

    \tkzDrawCircle[fill=white](O,Z)

    \tkzClipCircle(O,Z)


\pic[fill=orange,angle radius=0.9cm,,opacity=.3] {angle = A--p--C};
   \node at (0.05,0.3) {$\frac{\theta_2}{2}$} ;

   \tkzDefCircle[orthogonal through=A and C](O,Z) \tkzGetPoint{M1}
\tkzDrawCircle[fill,color=white](M1,A)
    \tkzDefCircle[orthogonal through=A and p](O,Z) \tkzGetPoint{M2}
\tkzDrawCircle[fill,color=white](M2,A)
     \tkzDefCircle[orthogonal through=C and p](O,Z) \tkzGetPoint{M3}
\tkzDrawCircle[fill, color=white](M3,C)

    \tkzDrawCircle(O,Z)

  \tkzDefCircle[orthogonal through=A and p](O,Z) \tkzGetPoint{L1}
\tkzDrawCircle[color=red](L1,A)
 
  \tkzDefCircle[orthogonal through=C and p](O,Z) \tkzGetPoint{L2}
\tkzDrawCircle[color=blue](L2,C)

 \tkzDefCircle[orthogonal through=A and E](O,Z) \tkzGetPoint{L3}
\tkzDrawCircle[color=blue](L3,A)

    \tkzDefCircle[orthogonal through=E and D](O,Z) \tkzGetPoint{L4}
\tkzDrawCircle[color=blue](L4,E)

\tkzDefCircle[orthogonal through=D and C](O,Z) \tkzGetPoint{L5}
\tkzDrawCircle[color=blue](L5,D)
    
    \node at (-0.25,0.25) {$\delta$} ;
    \node at (-0.55,-0.3) {$\frac{L_1}{2}$} ;
    \node at (0.3,-0.5) {$\frac{c}{2}$} ;

    \tkzDrawPoints[color=black,fill=red,size=5](p)
\end{tikzpicture}
\caption{Hyperbolic pentagon with all but one right angle.} \label{Hyperbolic pentagon}
\end{figure}
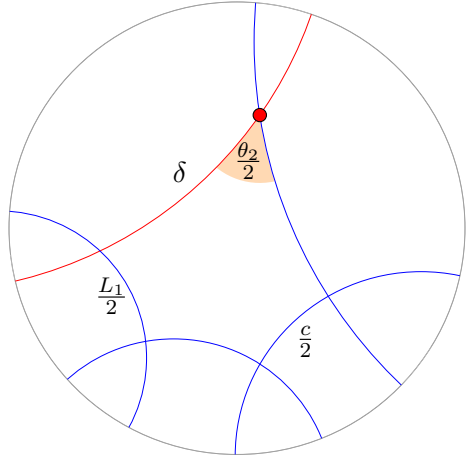
A hyperbolic pentagon formula from \cite[Example 2.2.7, (iv), (v)]{BusGeo} gives 
\begin{align}
\cosh(\delta)&=
 \frac{1}{\sin\left(\frac{\theta_2}{2}\right)\sinh(\frac{L_1}{2})}\sqrt{\cos^2\left(\frac{\theta_2}{2}\right)+2\cos(\frac{\theta_2}{2})\cosh(\frac{L_1}{2})+ \cosh^2\left(\frac{L_1}{2}\right)+K} \nonumber \\
&\geq\frac{\cos(\frac{\theta_2}{2})+ \cosh(\frac{L_1}{2})}{\sin(\frac{\theta_2}{2})\sinh(\frac{L_1}{2})}  
> \frac{\cosh(\frac{L_1}{2})}{\sinh(\frac{L_1}{2})}=\coth(\frac{L_1}{2})>1 \label{pentagon computation line}
\end{align}
where $K=2\cos(\frac{\theta_2}{2})\cosh\left(\frac{L_1}{2}\right)(\cosh\left(\frac{c}{2}\right)-1)+\cosh^2\left(\frac{c}{2}\right)-1>0$.
This means that $\delta>\varepsilon(L_1)$, for some positive function $\varepsilon$ of the length of the boundary geodesic. \\
\textit{\underline{Case 3. $L_1,L_2\in i[0,\pi)$:}}\\ \indent
Let $\theta_1=|L_1|$ and $\theta_2=|L_2|$ be the cone angles defined by $L_1$ and $L_2$, respectively. Cut along $\gamma'$ and the other two seams of the hyperbolic pair-of-pants, to obtain two hyperbolic quadrilaterals, each with $2$ right angles and two angles $\frac{\theta_1}{2}$ and $\frac{\theta_2}{2}$ at the cone points. One such quadrilateral is shown in Figure \ref{Hyperbolic quad2}.
\begin{figure}[!htb] \centering
\begin{tikzpicture}[scale=3]
    \tkzDefPoint(0,0){O}
    \tkzDefPoint(1,0){Z}

    \tkzDefPoint(-0.6,-0.1){q}
    \tkzDefPoint(0.4,-0.3){C}
    \tkzDefPoint(0.1,0.5){p}
    \tkzDefPoint(0.2,-0.5){D}
    \tkzDefPoint(0.1,-0.9){E}

    \tkzDrawCircle[fill=white](O,Z)

    \tkzClipCircle(O,Z)

\pic[fill=orange,angle radius=0.9cm,opacity=.3] {angle = q--p--C};
  \node at (0.05,0.3) {$\frac{\theta_2}{2}$} ;
    
    \pic[fill=orange,angle radius=0.9cm,opacity=.3] {angle = D--q--p};
 \node at (-0.25,-0.05) {$\frac{\theta_1}{2}$} ;

      \tkzDefCircle[orthogonal through=q and D](O,Z) \tkzGetPoint{M1}
\tkzDrawCircle[fill,color=white](M1,q)
         \tkzDefCircle[orthogonal through=q and p](O,Z) \tkzGetPoint{M2}
\tkzDrawCircle[fill,color=white](M2,q)
         \tkzDefCircle[orthogonal through=C and p](O,Z) \tkzGetPoint{M3}
\tkzDrawCircle[fill,color=white](M3,C)

    \tkzDrawCircle(O,Z)

     \tkzDefCircle[orthogonal through=q and p](O,Z) \tkzGetPoint{L1}
\tkzDrawCircle[color=red](L1,q)

     \tkzDefCircle[orthogonal through=C and p](O,Z) \tkzGetPoint{L2}
\tkzDrawCircle[color=blue](L2,C)

     \tkzDefCircle[orthogonal through=q and D](O,Z) \tkzGetPoint{L3}
\tkzDrawCircle[color=blue](L3,q)

         \tkzDefCircle[orthogonal through=D and C](O,Z) \tkzGetPoint{L4}
\tkzDrawCircle[color=blue](L4,D)

     \node at (-0.25,0.25) {$\delta$};

    \tkzDrawPoints[color=black,fill=red,size=5](p)
    \tkzDrawPoints[color=black,fill=red,size=5](q)
    
\end{tikzpicture}
\caption{Hyperbolic quadrilateral with two right angles and two non-right angles at cone points.} \label{Hyperbolic quad2}
\end{figure}
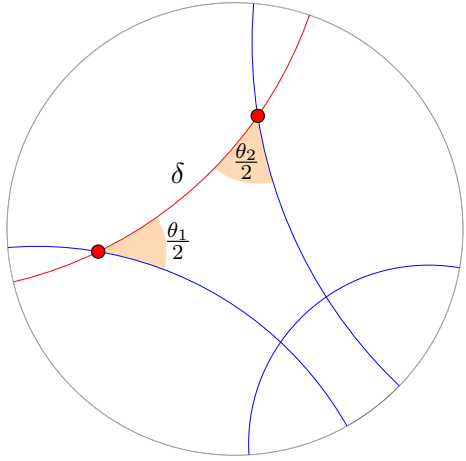
For any hyperbolic quadrilateral the formula $\cosh(\frac{c}{2})=\sin(\frac{\theta_1}{2})\sin(\frac{\theta_2}{2})\cosh(\delta)-\cos(\frac{\theta_1}{2})\cos(\frac{\theta_2}{2})$ \cite[Theorem 2.3.1, (i)]{BusGeo} gives  
\begin{align}
        \cosh(\delta)&=\frac{\cosh(\frac{c}{2})+\cos(\frac{\theta_1}{2})\cos(\frac{\theta_2}{2})}{\sin(\frac{\theta_1}{2})\sin(\frac{\theta_2}{2})}  \geq\frac{1+\cos(\frac{\theta_1}{2})\cos(\frac{\theta_2}{2})}{\sin(\frac{\theta_1}{2})\sin(\frac{\theta_2}{2})}>1 \label{quadrilateral}
    \end{align}
where the last inequality uses $\theta_1+\theta_2<2\pi$. This means that $\delta>\varepsilon(\theta_1,\theta_2)$, for some positive function $\varepsilon$ of the angles of the cone points.  The lower bound \eqref{quadrilateral} holds more generally for any $\theta_1+\theta_2<2\pi$, since this is enough for the hyperbolic quadrilateral to exist.  This arises below in Lemma~\ref{angle condition repulsion}.  
\end{proof}
\begin{remark}
Lemma~\ref{length and small angle repulsion} actually holds under slightly stronger conditions, by also allowing at most one cone angle to be equal to $\pi$. 
\end{remark}

\begin{remark} \label{except}
The lower bound \eqref{quadrilateral} is a universal lower bound on the distance between cone angles in any hyperbolic surface, {\em except} for the special case of a sphere with exactly three cone angles, or equivalently a pair of pants with cone angle boundary components, obtained by gluing together two copies of a hyperbolic triangle.  In this special case, the distance between the cone angles tends to zero as the area of the sphere, or hyperbolic triangle, tends to zero.  This is shown in Figure~\ref{meeting of dashed lines}.
\end{remark}

\begin{remark}
As a cone point or boundary geodesic acquires the behaviour of a cusp, $L_1\rightarrow 0$ or $\theta_1\rightarrow 0$ in the lower bounds \eqref{ineq1}, \eqref{pentagon computation line}  and \eqref{quadrilateral}  in Lemma~\ref{length and small angle repulsion} which implies that $\cosh{\delta}$ is unbounded as expected.  In this case one can instead measure distance between a horocycle around the cusp and any boundary component and produce a no-go zone outside a suitably chosen horocycle.  
\end{remark}

In the case of cone angles $L_j\in i[0,2\pi)$ where some angles are greater than $\pi$, we must use a different method to extend the result of Lemma \ref{length and small angle repulsion}.  We prove that there exists a  positive lower bound on distances between cone points, again with the exception of the pair of pants construction described in Remark~\ref{except}.  
\begin{lemma} \label{angle condition repulsion} 
Given lengths $L_j\in i[0,2\pi)$, if $|L_j|+|L_k|<2\pi$ for some $j,k\in\{1,\ldots,n\}$, then there exists a positive lower bound between the cone points $p_j$ and $p_k$. 
\end{lemma}
\begin{proof}

Define $\theta_j=|L_j|$.  Without loss of generality we assume that $(j,k)=(1,2)$, $\theta_1\geq\theta_2$ and furthermore that 
\[\theta_1>\pi\] 
so that $\theta_2<\pi$.  Since, otherwise, if $\theta_1\leq\pi$ then Lemma~\ref{length and small angle repulsion} applies to produce the positive lower bound.

Suppose that $p_1$ and $p_2$ are sufficiently close, and for the moment assume that no other cone points are nearby.   Then we may represent a disk neighbourhood $\cD$ of $p_1$ and $p_2$ locally via a region in the hyperbolic plane as follows. We model $\cD$ as a region in the hyperbolic  plane, as depicted in Figure~\ref{planar small cone angles}, 
   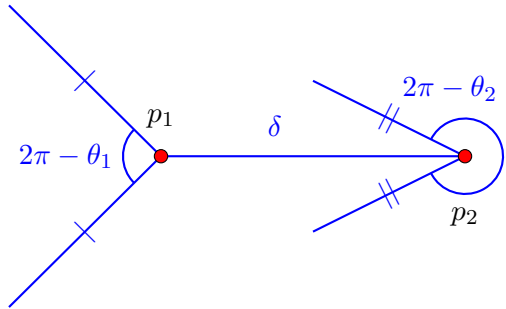
\begin{figure}[htb!]
   \centering
    \begin{tikzpicture}
    \draw[blue, thick]
    (-2,2) coordinate (a) node[black,left]  {}
    --node[sloped] {$|$} (0,0) coordinate (b) node[black,above] {} 
    --node[sloped] {$|$} (-2,-2) coordinate (c) node[black,right] {}
    pic[ draw=blue, -, angle eccentricity=1.2, angle radius=0.5cm]
    {angle=a--b--c};
    \node[text=blue] at (-1.25,0) {$2\pi-\theta_1$} ;
    
    \draw[blue, thick]
    (2,1) coordinate (e) node[black,left]  {}
    --node[sloped] {$||$} (4,0) coordinate (f) node[black,above] {} 
    --node[sloped] {$||$} (2,-1) coordinate (g) node[black,right] {}
    pic[ draw=blue, -, angle eccentricity=1.2, angle radius=0.5cm]
    {angle=g--f--e};
    \draw[blue,thick] (b) -- (f) ;
    \node[text=blue] at (1.5,0.4) {$\delta$};
    \node[text=blue] at (3.8,0.9) {$2\pi-\theta_2$};

        \tkzDrawPoints[color=black,fill=red,size=5](b) ;
        \node[text=black] at (0,0.5)  {$p_1$};
        \tkzDrawPoints[color=black,fill=red,size=5](f) ;
        \node[text=black] at (4,-0.8)  {$p_2$};
\end{tikzpicture}  
\captionsetup{skip=0.9cm}
 \caption{A local depiction in the hyperbolic plane of two nearby cone points $p_1$ and $p_2$.}
 \label{planar small cone angles}
\end{figure}
where geodesic rays from the cone points are identified, so that the marked angles of $2\pi-\theta_1$ and $2\pi-\theta_2$ are essentially deleted from the plane.  Note that geodesics appear as straight lines in the diagram.  Let $\delta$ be the separation of $p_1$ and $p_2$. If $\delta$ is small enough, since $\theta_1+\theta_2<2\pi$, there must be a point in which the geodesic rays corresponding to $p_1$ meet with those of $p_2$  as depicted in Figure~\ref{meeting of dashed lines}, essentially due to the Euclidean picture which approximates the hyperbolic picture at small enough scales.  Figure~\ref{meeting of dashed lines} can be constructed by gluing along its edges two identical hyperbolic triangles with interior angles $\theta_1/2$, $\theta_2/2$ and $\alpha$ for any $\alpha\in(0,\pi-\theta_1/2-\theta_2/2)$.
   \begin{figure}[htb!]
   \centering
    \begin{tikzpicture}[scale=1.5]
    \draw[blue, thick]
    (-2.1,2.1) coordinate (a) node[black,left]  {}
    --node[sloped] {$|$} (0,0) coordinate (b) node[black,above] {} 
    --node[sloped] {$|$} (-2.1,-2.1) coordinate (c) node[black,right] {}
    pic[ draw=blue, -, angle eccentricity=1.2, angle radius=0.5cm]
    {angle=a--b--c};
    \node[text=blue] at (-0.9,0) {$2\pi-\theta_1$} ;
    
    \draw[blue, thick]
    (-1,1) coordinate (e) node[black,left]  {}
    --node[sloped] {$||$} (1,0) coordinate (f) node[black,above] {} 
    --node[sloped] {$||$} (-1,-1) coordinate (g) node[black,right] {}
    pic[ draw=blue, -, angle eccentricity=1.2, angle radius=0.5cm]
    {angle=g--f--e};
    \draw[blue,thick] (b) -- (f) ;
    \node[text=blue] at (0.35,0.125) {$\delta$};
    \node[text=blue] at (1.2,0.5) {$2\pi-\theta_2$};

        \tkzDrawPoints[color=black,fill=red,size=5](b) ;
        \node[text=black] at (0,0.25)  {$p_1$};
        \tkzDrawPoints[color=black,fill=red,size=5](f) ;
        \node[text=black] at (1.2,-0.5)  {$p_2$};
\end{tikzpicture} 
 \captionsetup{skip=0.9cm}
 \caption{A local depiction in the hyperbolic plane of two cone points $p_1$ and $p_2$ whose angle sum is less than $2\pi$ arbitrarily close to one another.} \label{meeting of dashed lines}
\end{figure}
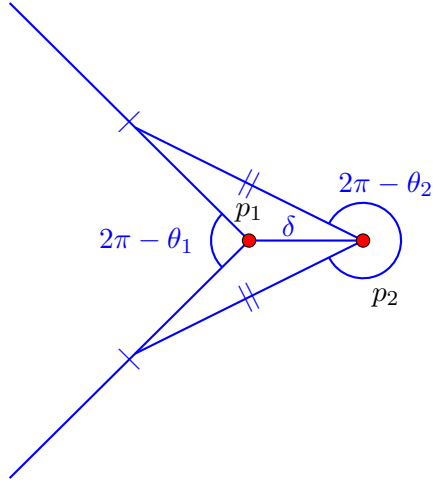
As $\delta$ increases, the intersection point moves out to infinity and the endpoints of the geodesic rays on the circle at infinity move together.  Then they pass through each so that the geodesic rays do not meet, and a shortest path between them, given by a geodesic meeting the rays perpendicularly, forms and moves towards $p_1$ and $p_2$ as the endpoints move apart.  In fact, due to identification of the geodesic rays, the perpendicular geodesic represents a simple closed geodesic surrounding $p_1$ and $p_2$ in the surface $X$. This is shown in red in Figure~\ref{surrounding angle geodesic}. 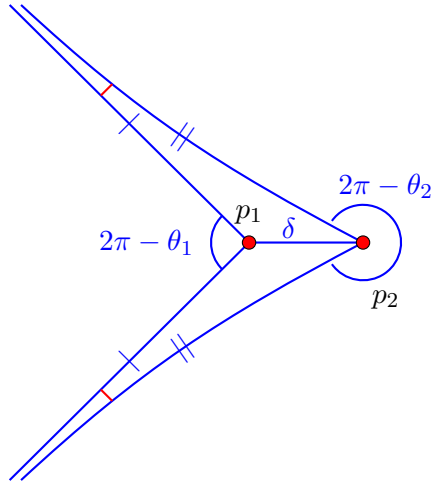
\begin{figure}[htb!]
   \centering
\begin{tikzpicture}[scale=1.5]
    \draw[blue, thick]
    (-2.1,2.1) coordinate (a) node[black,left]  {}
    to node[sloped] {$|$} (0,0) coordinate (b) node[black,above] {} 
    --node[sloped] {$|$} (-2.1,-2.1) coordinate (c) node[black,right] {}
    pic[ draw=blue, -, angle eccentricity=1.2, angle radius=0.5cm]
    {angle=a--b--c};
    \node[text=blue] at (-0.9,0) {$2\pi-\theta_1$} ;
    
    \draw[blue, thick]
    (-2,2.1) coordinate (e) node[black,left]  {}
    to[in=152, out=-43] node[sloped] {$||$} (1,0) coordinate (f) node[black,above] {} 
    to[in=43, out=208] node[sloped] {$||$} (-2,-2.1) coordinate (g) node[black,right] {}
    pic[ draw=blue, -, angle eccentricity=1.2, angle radius=0.5cm]
    {angle=g--f--e};
    \draw[blue,thick] (b) -- (1,0) ;
    \node[text=blue] at (0.35,0.125) {$\delta$};
    \node[text=blue] at (1.2,0.5) {$2\pi-\theta_2$};

     \draw[red, thick] (-1.3,1.3) to (-1.2,1.4);
     \draw[red, thick] (-1.3,-1.3) to (-1.2,-1.4);

        \tkzDrawPoints[color=black,fill=red,size=5](b) ;
        \node[text=black] at (0,0.25)  {$p_1$};
        \tkzDrawPoints[color=black,fill=red,size=5](f) ;
        \node[text=black] at (1.2,-0.5)  {$p_2$};
\end{tikzpicture} 
 \captionsetup{skip=0.9cm}
 \caption{A local depiction in the hyperbolic plane of two nearby cone points $p_1$ and $p_2$ with angle sum less than $2\pi$, surrounded by a simple closed geodesic given by two red arcs.} \label{surrounding angle geodesic}
\end{figure}

We learn two important ideas from the previous construction.  Firstly, when $\delta$ is sufficiently small, a hyperbolic sphere with three cone points appears around $p_1$ and $p_2$ producing the pair of pants construction described in Remark~\ref{except}.  But such a situation cannot occur except if the ambient surface is a pair of pants, as it would imply the existence of a disconnected closed region in the surface.  Hence there is a positive lower bound on the distance between the two cone points.

Secondly, when cone points $p_1$ and $p_2$ whose angle sum is less than $2\pi$ are sufficiently close to produce the local model in Figure~\ref{surrounding angle geodesic}, then there exists a disk neighbourhood of $p_1$ and $p_2$ with boundary a simple closed geodesic.  In particular, this holds when  $\theta_1>\pi$, which contrasts with the behaviour when both $\theta_1$ and $\theta_2$ are less than $\pi$, covered by Lemma~\ref{length and small angle repulsion}, in which case 
arbitrarily far $p_1$ and $p_2$ are contained in a geodesically bounded disk neighbourhood.  The simple closed geodesic and two cone points bound a hyperbolic pair of pants formed from gluing two identical quadrilaterals as in Case 3 of the proof of Lemma~\ref{length and small angle repulsion}.  The inequality \eqref{quadrilateral} applies to angles satisfying $\frac12(\theta_1+\theta_2)<\pi$ and gives a positive lower bound on the distance $\delta$ from $p_1$ to $p_2$.  

A less geometric proof of the existence  of Figures~\ref{meeting of dashed lines} and \ref{surrounding angle geodesic} and hence the existence of the lower bound, uses the following continuity argument.  The local holonomy of the cone points $p_1$ and $p_2$ angles $\theta_1$ and $\theta_2$ is represented by rotation matrices $S$ and $R$ which lie in conjugate $SO(2)$ subgroups of $SL(2,\br)$.  For large enough $\delta$, the holonomy of a loop surrounding $p_1$ and $p_2$ is hyperbolic.  As $\delta\to 0$, $S$ and $R$ tend towards the same $SO(2)$ subgroup of $SL(2,\br)$ and their product tends towards a rotation matrix of angle $\theta_1+\theta_2$ which is elliptic.  By the intermediate value theorem (applied to the trace), for some $\delta$ parabolic holonomy, i.e.\ a cusp, occurs.

When there are more than two cone points close together, if the local model now contains cone points inside the regions in Figures~\ref{meeting of dashed lines} and \ref{surrounding angle geodesic} then the intersection of the geodesic rays occurs earlier, i.e.\ for larger distance $\delta$ from $p_1$ to $p_2$, since any cone angle causes the geodesic rays to bend towards each other.  Thus the lower bound for the distance $\delta$ is greater than that calculated above.  Alternatively, the continuity argument can be applied to several cone angles approaching a single point and the product of several holonomy matrices to produce an elliptic holonomy.  

\end{proof}
If two cone points $p_j$ and $p_k$ have cone angles which satisfy $\theta_j+\theta_k>2\pi$ then there is no longer a lower bound on the distance between these points.  The following Lemma proves a more general statement which shows that a collection of cone points may merge.
\begin{lemma}\cite{MZhCon}
\label{angles merging}
Given $\cM_{g,n}^{\mathrm{hyp}}({\bf L})$ with ${\bf L}\in i[0,2\pi)^n$ and $S\subset\{1,\ldots,n\}$ satisfying $\sum_{j\in S}|L_j|>2\pi (|S|-1)$, there does not exist a lower bound on distances between cone points.  The cone points $p_j$ with cone angles $|L_j|$ for $j\in S$ may merge to form a single cone point of cone angle $|L_{S}|=\sum_{j\in S}|L_j|-2\pi (|S|-1)$.
\end{lemma}
\begin{proof}
Define ${\bm{\theta}}=-i\mathbf{L}\in [0,2\pi)^n$ and let $p_j$ for $j\in S$ denote cone points satisfying the hypothesis of the lemma.   A conical hyperbolic metric can be continuously deformed so that the $p_j$ merge with resulting cone angle of $\theta_{S}=\sum_{j\in S}\theta_j-2\pi (|S|-1)$.  This is proven in \cite{MZhCon}.  The resulting cone angle arises by summing the complementary angles in the local picture, as shown for two points in Figure~\ref{planar merging cone angles}, or  via the following local form of conical hyperbolic metrics.
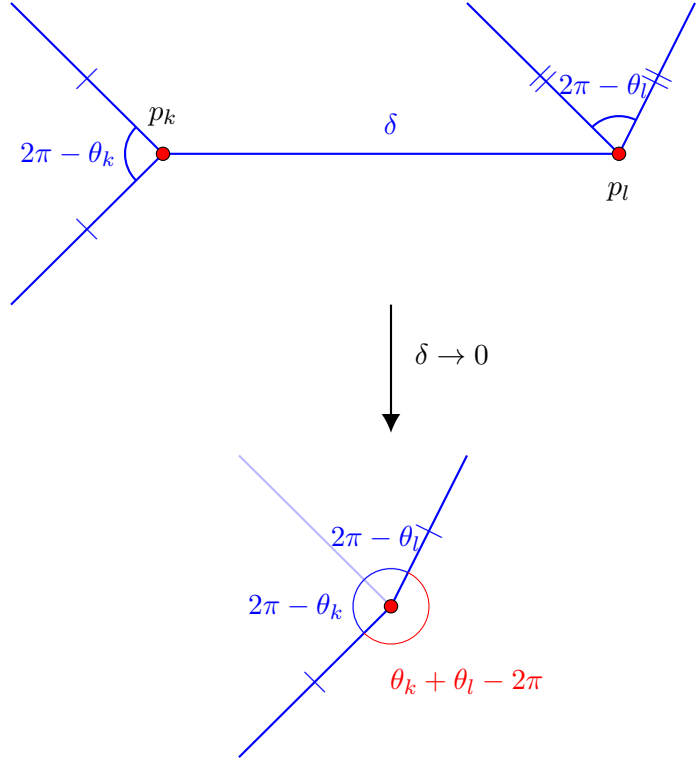
\begin{figure}[htb!]
   \centering
    \begin{tikzpicture}
    \draw[blue, thick]
    (-2,2) coordinate (a) node[black,left]  {}
    --node[sloped] {$|$} (0,0) coordinate (b) node[black,above] {} 
    --node[sloped] {$|$} (-2,-2) coordinate (c) node[black,right] {}
    pic[ draw=blue, -, angle eccentricity=1.2, angle radius=0.5cm]
    {angle=a--b--c};
    \node[text=blue] at (-1.25,0) {$2\pi-\theta_k$} ;
    
    \draw[blue, thick]
    (4,2) coordinate (e) node[black,left]  {}
    --node[sloped] {$||$} (6,0) coordinate (f) node[black,above] {} 
    --node[sloped] {$||$} (7,2) coordinate (g) node[black,right] {}
    pic[ draw=blue, -, angle eccentricity=1.2, angle radius=0.5cm]
    {angle=g--f--e};
    \draw[blue,thick] (b) -- (f) ;
    \node[text=blue] at (3,0.4) {$\delta$};
    \node[text=blue] at (5.8,0.9) {$2\pi-\theta_l$};
    
    \draw[thick,->] (3,-2) -- (3,-3.7);
    \node[text=black] at (3.8,-2.65)  {$\delta \rightarrow 0$};
    
    \draw[blue, thick,opacity=0.15]
    (1,-4) coordinate (k) node[black,left]  {}
    -- (3,-6) coordinate (l) node[black,above] {} 
    --(1,-8) coordinate (m) node[black,right] {}
    pic[ draw=blue, -, angle eccentricity=1.2, angle radius=0.5cm]
    {angle=k--l--m};
    \node[text=blue] at (1.75,-6) {$2\pi-\theta_k$} ;
    
    \draw[blue, thick,opacity=0.15]
    (1,-4) coordinate (r) node[black,left]  {}
    -- (3,-6) coordinate (s) node[black,above] {} 
    -- (4,-4) coordinate (t) node[black,right] {}
    pic[ draw=blue, -, angle eccentricity=1.2, angle radius=0.5cm]
    {angle=t--s--r};

    \pic[ draw=blue, -, angle eccentricity=1.2, angle radius=0.5cm]
    {angle=k--l--m};
    \pic[ draw=blue, -, angle eccentricity=1.2, angle radius=0.5cm]
    {angle=t--s--r};
    \draw[blue, thick](3,-6) coordinate (s) node[black,above] {} 
    --node[sloped] {$|$} (4,-4) coordinate (t) node[black,right] {};
    \draw[blue, thick](3,-6) (3,-6) coordinate (l) node[black,above] {} 
    --node[sloped] {$|$} (1,-8) coordinate (m) node[black,right] {};
    
    \node[text=blue] at (2.8,-5.1) {$2\pi-\theta_l$};
    \pic[draw=red, -, angle eccentricity=1.2, angle radius=0.5cm]
    {angle=m--s--t};
    \node[text=red] at (4,-7) {$\theta_k+\theta_l-2\pi$}  ;

        \tkzDrawPoints[color=black,fill=red,size=5](b) ;
        \node[text=black] at (0,0.5)  {$p_k$};
        \tkzDrawPoints[color=black,fill=red,size=5](f) ;
        \node[text=black] at (6,-0.5)  {$p_l$};
        \tkzDrawPoints[color=black,fill=red,size=5](l) ;
\end{tikzpicture}
  \captionsetup{skip=0.9cm}
 \caption{A depiction in the hyperbolic plane of two cone angles $p_k$ and $p_l$ merging together.} \label{planar merging cone angles}
\end{figure}
 Around a cone point $p_k$ with cone angle $\theta_k\in(0,2\pi)$ the metric takes the form $g=f(z_k)|z_k|^{2\left(\frac{\theta_k}{2\pi}-1\right)}|dz_k|^2$, where $z_k$ is a local conformal coordinate and $f$ is a smooth positive function \cite{McOPoi, MonPoi}. Suppose that cone points $p_j$ are sufficiently close together, so that they exist in some local neighbourhood $\cD$ with local conformal coordinate $z$. Suppose also that in $\cD$ the point $p_j$ is at $z=\zeta_j$. Locally, the metric takes the form \begin{equation*}
g_\cD=f(z)\left(\prod_{j\in S}|z-\zeta_j|^{2\left(\frac{\theta_j}{2\pi}-1\right)}\right)|dz|^2,
\end{equation*}
and in the limit at $\zeta_j\to 0$ this approaches  \begin{equation} \label{limit metric}
f(z)|z|^{\sum_{j\in S}2\left(\frac{\theta_j}{2\pi}-1\right)}|dz|^2=f(z)|z|^{2\left(\frac{\theta_S}{2\pi}-1\right)}|dz|^2.
\end{equation}
Since by assumption $\sum_{j\in S}\theta_j>2\pi (|S|-1)$ with each $\theta_j<2\pi$, we deduce that $\theta_S\in (0,2\pi)$ and \eqref{limit metric} gives the form of a conical hyperbolic metric with cone angle $\theta_S$. 
\end{proof}

\noindent  The following result gives a converse to Lemma~\ref{angles merging}.
 \begin{lemma}
\label{converse}
For $\cM_{g,n}^{\mathrm{hyp}}(L_1,\ldots,L_n)$ with ${\bf L}=i{\bm{\theta}}\in i[0,2\pi)^n$, if $S\subset\{1,\ldots,n\}$ satisfies $\sum_{j\in S}\theta_j<2\pi (|S|-1)$ then there exists a lower bound on the diameter of a neighbourhood containing the cone points $\{p_j\}_{j\in S}$.
\end{lemma}
\begin{proof}
When $|S|=2$, then there exists cone points $p_j$ and $p_k$ with cone angles satisfying $\theta_j+\theta_k<2\pi$ in which case the proof of Lemma~\ref{angle condition repulsion} shows that there exists a lower bound on the distance between $p_j$ and $p_k$, hence a lower bound on the diameter of a neighbourhood.  For $|S|>2$ we can proceed by induction.  If $\sum_{j\in S}\theta_j<2\pi (|S|-1)$ and suppose first that there exists $k,\ell\in S$ with $\theta_k+\theta_\ell<2\pi$, then there exists a lower bound on the distance between $p_k$ and $p_\ell$, hence a lower bound on the diameter of a neighbourhood containing the cone points $\{p_j\}_{j\in S}$.  Thus we may assume that $\theta_k+\theta_\ell>2\pi$ for all $k,\ell\in S$.  Suppose there is a path in $\cM_{g,n}^{\mathrm{hyp}}(L_1,\ldots,L_n)$ so that the cone points $\{p_j\}_{j\in S}$ merge.  Consider the closest two points, say $p_k$ and $p_\ell$.  They can merge to form a new cone point $p_{k\ell}$ of angle $\theta_{k\ell}=\theta_k+\theta_\ell-2\pi$.  Just before they merge, the holonomy around a path containing $p_k$ and $p_\ell$ is arbitrarily close to the holonomy around the cone point $p_{k\ell}$, and in particular elliptic.  So we can treat these two points like the cone point $p_{k\ell}$ and reduce the problem to the $|S|-1$ cone points $\{p_j\}_{j\in S\backslash\{k,\ell\}}\cup \{p_{k\ell}\}$ labeled by $S'=S\backslash\{k,\ell\}\cup\{k\ell\}$ that satisfy $\sum_{j\in S'}\theta_j<2\pi (|S'|-1)$.  By induction, there exists a lower bound on the diameter of a neighbourhood containing these cone points. 
\end{proof}

The following interaction between boundary geodesics and cone angles greater than $\pi$ shows that non-nodal hyperbolic surfaces may arise as limits of families. 
\begin{lemma} \label{angle length merge}
Consider $\cM_{g,n}^{\mathrm{hyp}}(L_1,\ldots,L_n)$, with $L_j\in i(\pi,2\pi)$ and $L_k\in\bbR_{>0}$ sufficiently large for some $k\neq j$. There does not exists a lower bound on distances between the corresponding cone point and boundary geodesic.  
\end{lemma}
\begin{proof}
Define $\theta_j=-iL_j\in (\pi,2\pi)$. Let $p_j$ and $\beta_k$ denote a cone point and a geodesic boundary component, respectively, which fit the description given in the lemma. First, suppose that $(g,n)\neq(0,3)$, so we may find two geodesic curves $\gamma_1$ and $\gamma_2$, which are not homotopic to each other (other than the $g=1$ case, where they are equal) or $\beta_k$, and together with $\beta_k$ form a hyperbolic pair-of-pants $P$ with $p_j$ in its interior. Cut along the seams of $P$ along with unique geodesics from $p_j$ meeting the three boundary geodesics perpendicularly, to decompose $P$ into three geodesic pentagons and one geodesic hexagon, as depicted in Figure~\ref{cutting into pentagons and hexagon}.
\begin{figure}[ht]  
	\centerline{\includegraphics[height=6.5cm]{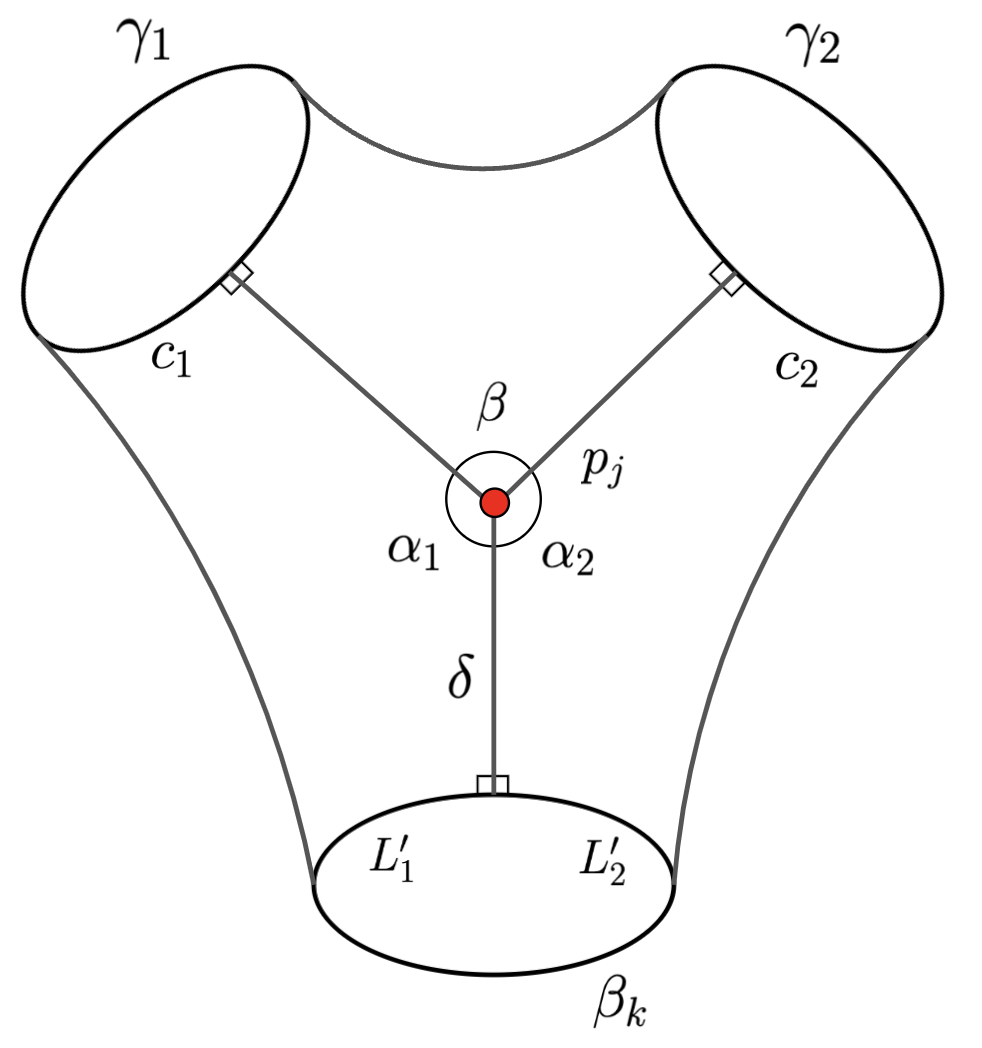}}
	\caption{Hyperbolic pair-of-pants with a single cone point of cone angle greater than $2\pi$.}
	\label{cutting into pentagons and hexagon} 
\end{figure} 
    Here, $\alpha_1+\alpha_2+\beta=\theta$ and $\delta$ is the distance between $p_j$ and $\beta_k$. Consider one of the two pentagons with $\delta$ as a side length. For $i=1$ or $2$, let $L_i' \in (0,L_k)$ denote the side length of the pentagon coming from $\beta_k$ and let $c_i$ denote the relevant length component of $\gamma_i$ which makes up the opposite side length to $\beta_k$. This pentagon is depicted in Figure~\ref{pentagon for crowned surfaces}. 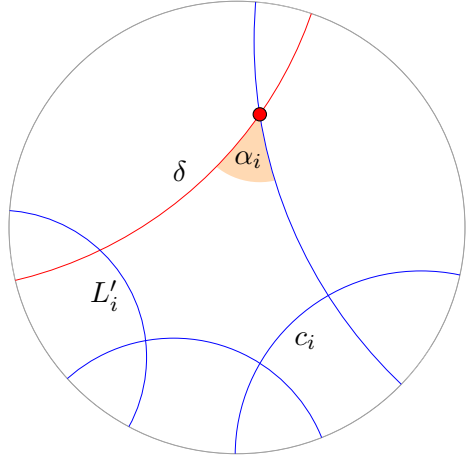
\begin{figure}[!htb] \centering
\begin{tikzpicture}[scale=3]
    \tkzDefPoint(0,0){O}
    \tkzDefPoint(1,0){Z}

    \tkzDefPoint(-0.6,-0.1){A}
    \tkzDefPoint(0.4,-0.3){C}
    \tkzDefPoint(0.1,0.5){p}
    \tkzDefPoint(0.1,-0.6){D}
    \tkzDefPoint(-0.4,-0.5){E}

    \tkzDrawCircle[fill=white](O,Z)

    \tkzClipCircle(O,Z)

 \pic[fill=orange,angle radius=0.9cm,,opacity=.3] {angle = A--p--C};
  \node at (0.05,0.3) {$\alpha_i$} ;

          \tkzDefCircle[orthogonal through=A and C](O,Z) \tkzGetPoint{M1}
\tkzDrawCircle[fill,color=white](M1,q)
          \tkzDefCircle[orthogonal through=A and p](O,Z) \tkzGetPoint{M2}
\tkzDrawCircle[fill,color=white](M2,A)
       \tkzDefCircle[orthogonal through=C and p](O,Z) \tkzGetPoint{M3}
\tkzDrawCircle[fill,color=white](M3,C)

    \tkzDrawCircle(O,Z)

    \tkzDefCircle[orthogonal through=A and p](O,Z) \tkzGetPoint{L1}
\tkzDrawCircle[color=red](L1,A)
  \tkzDefCircle[orthogonal through=C and p](O,Z) \tkzGetPoint{L2}
\tkzDrawCircle[color=blue](L2,C)
   \tkzDefCircle[orthogonal through=A and E](O,Z) \tkzGetPoint{L3}
\tkzDrawCircle[color=blue](L3,A)
   \tkzDefCircle[orthogonal through=E and D](O,Z) \tkzGetPoint{L4}
\tkzDrawCircle[color=blue](L4,E)
   \tkzDefCircle[orthogonal through=D and C](O,Z) \tkzGetPoint{L5}
\tkzDrawCircle[color=blue](L5,D)

    \node at (-0.25,0.25) {$\delta$} ;
    \node at (-0.58,-0.3) {$L_i'$} ;
    \node at (0.3,-0.5) {$c_i$} ;

    \tkzDrawPoints[color=black,fill=red,size=5](p)
\end{tikzpicture}
\caption{Hyperbolic pentagon with all but one right angle which may be greater than $\frac{\pi}{2}$.} \label{pentagon for crowned surfaces}
\end{figure}
Any hyperbolic pentagon satisfies the following formula from \cite[Example 2.2.7, (iv), (v)]{BusGeo}  \begin{equation} \hspace{-0.1cm}
\sinh(\delta)=\frac{\cosh(c_i)+\cos(\alpha_i)\cosh(L_i')}{\sin(\alpha_i)\sinh(L_i')}. 
\end{equation}
By symmetry of the pair of pants, there exists elements of  $\cM_{g,n}^{\mathrm{hyp}}(L_1,\ldots,L_n)$ with $p_j$ positioned so that $\beta$ is less than both $\alpha_1$ and $\alpha_2$. On such hyperbolic surfaces, at least one of the $\alpha_1$ or $\alpha_2$ is greater than $\frac{\pi}{2}$, since $\theta>\pi$. Consider the subset $X\subset\cM_{g,n}^{\mathrm{hyp}}(L_1,\ldots,L_n)$  where $\alpha_2>\frac{\pi}{2}$, so that $\cos(\alpha_2)<0$. Suppose that $L_j$ is sufficiently large, so that $\cos(\alpha_2)\cosh(L_2')<-1$. Keeping $\alpha_1, \alpha_2, \beta$ and $L_2'$ constant, we may consider a path in $X$ where the length of $\gamma_2$ varies, so that $\cosh(c_2)$ approaches $|\cos(\alpha_2)\cosh(L_2')|>1$. This path will contain points in $X$ corresponding to hyperbolic surfaces with arbitrarily small separation length $\delta$. In this situation, the cone point $p_j$ and boundary geodesic $\beta_k$ will merge, rather than producing a cusp elsewhere in the surface, which would occur if $c_2\rightarrow 0$ were possible. \\ \indent
In the case $(g,n)=(0,3)$, surfaces are hyperbolic pairs of pants with two boundary components $\beta_1$ and $\beta_2$ and a single cone point $\beta_3=p$. These can be decomposed into two pentagons by cutting along seams and unique geodesics from $p$ meeting the two boundary geodesics perpendicularly. Each pentagon has all right angles except one, of angle $\alpha$, which is at $p$. We depict such a pentagon in Figure~\ref{pentagon for crowned surfacesz}. 
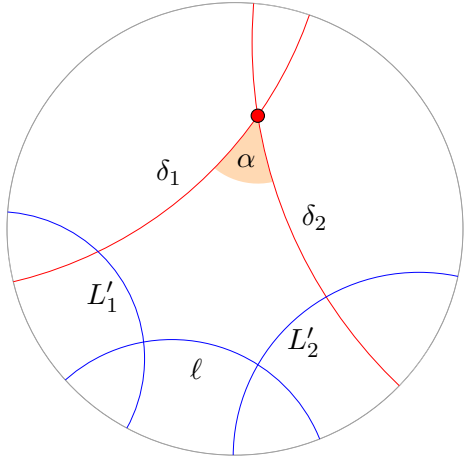
\begin{figure}[!htb] \centering
\begin{tikzpicture}[scale=3]
    \tkzDefPoint(0,0){O}
    \tkzDefPoint(1,0){Z}

    \tkzDefPoint(-0.6,-0.1){A}
    \tkzDefPoint(0.4,-0.3){C}
    \tkzDefPoint(0.1,0.5){p}
    \tkzDefPoint(0.1,-0.6){D}
    \tkzDefPoint(-0.4,-0.5){E}

    \tkzDrawCircle[fill=white](O,Z)

    \tkzClipCircle(O,Z)

 \pic[fill=orange,angle radius=0.9cm,,opacity=.3] {angle = A--p--C};
  \node at (0.05,0.3) {$\alpha$} ;

          \tkzDefCircle[orthogonal through=A and C](O,Z) \tkzGetPoint{M1}
\tkzDrawCircle[fill,color=white](M1,q)
          \tkzDefCircle[orthogonal through=A and p](O,Z) \tkzGetPoint{M2}
\tkzDrawCircle[fill,color=white](M2,A)
       \tkzDefCircle[orthogonal through=C and p](O,Z) \tkzGetPoint{M3}
\tkzDrawCircle[fill,color=white](M3,C)

    \tkzDrawCircle(O,Z)

    \tkzDefCircle[orthogonal through=A and p](O,Z) \tkzGetPoint{L1}
\tkzDrawCircle[color=red](L1,A)
  \tkzDefCircle[orthogonal through=C and p](O,Z) \tkzGetPoint{L2}
\tkzDrawCircle[color=red](L2,C)
   \tkzDefCircle[orthogonal through=A and E](O,Z) \tkzGetPoint{L3}
\tkzDrawCircle[color=blue](L3,A)
   \tkzDefCircle[orthogonal through=E and D](O,Z) \tkzGetPoint{L4}
\tkzDrawCircle[color=blue](L4,E)
   \tkzDefCircle[orthogonal through=D and C](O,Z) \tkzGetPoint{L5}
\tkzDrawCircle[color=blue](L5,D)

    \node at (-0.29,0.25) {$\delta_1$} ;
    \node at (0.35,0.05) {$\delta_2$} ;
    \node at (-0.17,-0.62) {$\ell$} ;
    \node at (-0.58,-0.3) {$L_1'$} ;
    \node at (0.3,-0.5) {$L_2'$} ;

    \tkzDrawPoints[color=black,fill=red,size=5](p)
\end{tikzpicture}
\caption{Hyperbolic pentagon constructed from a $(0,3)$ surface, with all but one right angle.} \label{pentagon for crowned surfacesz}
\end{figure}
Here $\ell$ denotes the length between boundary geodesics, which is variable within $\cM_{0,3}^{\mathrm{hyp}}(L_1,L_2,ip)$, and $L_1'$ and $L_2'$ are lengths of sections of $\beta_1$ and $\beta_2$ respectively. \\ \indent
Since $\theta>\pi$, one of the two possible pentagons must have $\alpha>\frac{\pi}{2}$. From here the proof is as in the general case, where we can find a path in $\cM_{0,3}^{\mathrm{hyp}}(L_1,L_2,ip)$ containing points which correspond to hyperbolic surfaces with arbitrarily small values of $\delta_2$. We note that in this case such a path is produced by fixing $L_1'$ and varying $\ell$, which will cause $L_2'$ to vary without changing $L_2$.   
\end{proof}

A consequence of Lemma~\ref{angle length merge} is the failure, for general $\mathbf{L}$, of the existence of a compactification of $\modm_{g,n}^{\mathrm{hyp}}(\mathbf{L})$ via the introduction of nodal hyperbolic surfaces.  We show this failure explicitly here.  Under the hypotheses of Lemma~\ref{angle length merge}, there exists a path in the moduli space  $\cM_{g,n}^{\mathrm{hyp}}(L_1,\ldots,L_n)$ with hyperbolic surfaces of decreasing distance between the given cone point and boundary component. In the limit of this path is a hyperbolic surface with piecewise geodesic boundary, where the cone point and boundary geodesic have merged. A diagram of this process is depicted in Figure~\ref{Formation of crowned surface component}. \begin{figure}[ht]  
	\centerline{\includegraphics[height=9cm]{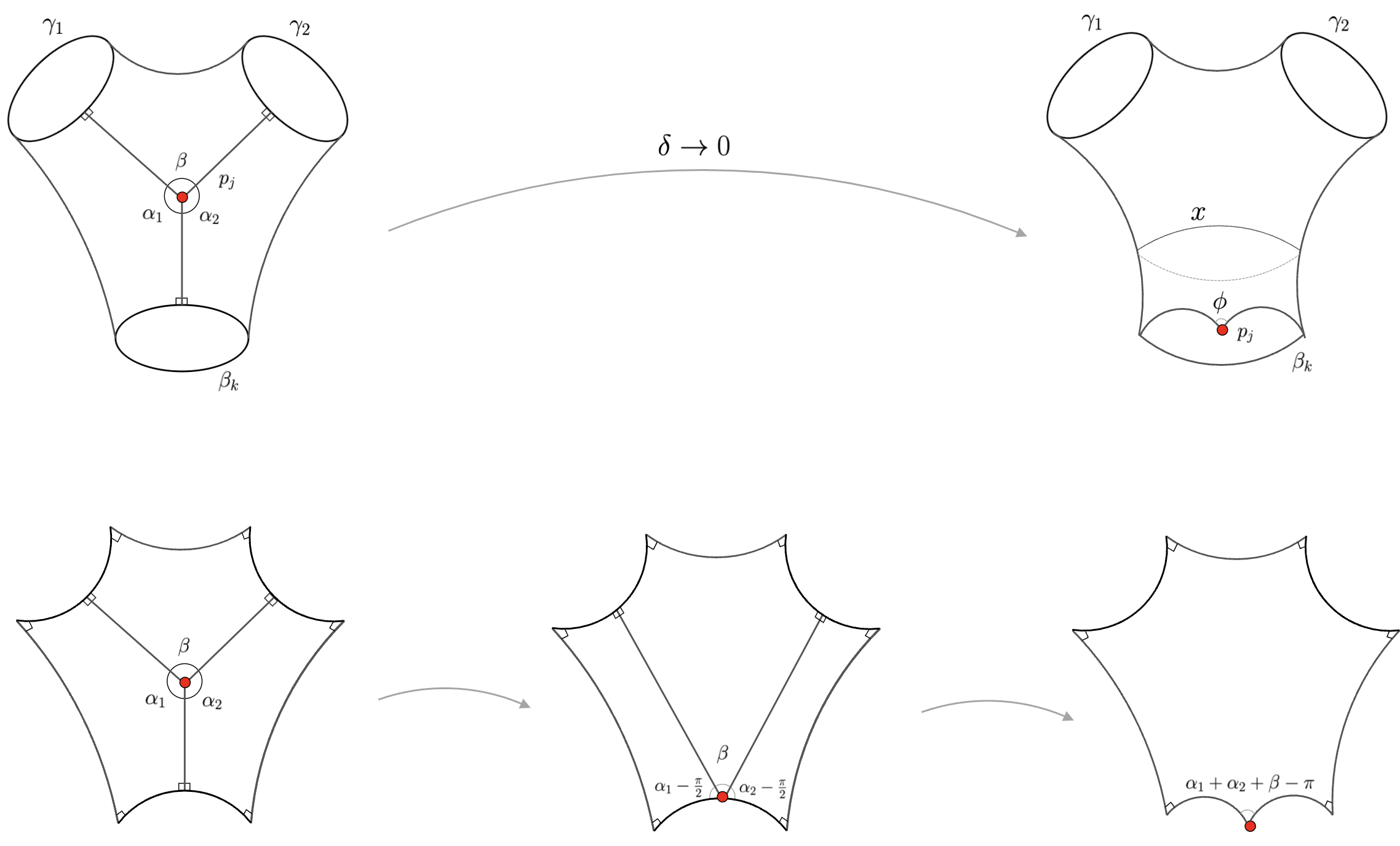}}
	\caption{Formation of piecewise geodesic boundary component as $\delta\rightarrow 0$.}
	\label{Formation of crowned surface component} 
\end{figure} 
Here the angle is $\phi=\theta_j-\pi$ at $p_j\in\beta_k$, where the piecewise geodesic boundary $\beta_k$ is not smooth. There exists a shorter simple closed geodesic path $\beta'$ which is homotopic to $\beta_k$. This violates the hyperbolic boundary holonomy, in the sense that now $2\cosh(L_k/2)\neq\left|\mathrm{tr}(A_k)\right|$ where $A_k\in PSL(2,\br)$ represents the conjugacy class defined by the holonomy of the metric around $\beta_k$. Instead, $\left|\mathrm{tr}(A_k)\right|=2\cosh(x/2)$ where $x=\ell(\beta')$ is the length of $\beta'$. As a function of $\theta_j$ and $L_k$ the length $x$ is uniquely determined by the equation
\begin{equation} \label{crownlength1}
\sin ^{2}\left(\phi\right) \cosh(L_k) \tanh ^{2}\left(\frac{x}{2}\right)=\cosh (x) \tanh ^{2}\left(\frac{L_k}{2}\right)-\cos ^{2}\left(\phi\right).
\end{equation}
As $\phi\rightarrow\pi$, the limits of each side of equation \eqref{crownlength1} are given by $\cosh(L_k) \tanh ^{2}\left(\frac{x}{2}\right)$ respectively $  \cosh (x) \tanh ^{2}(\frac{L_k}{2})$, implying $x\smallsim  L_k$ as we would expect. 
These surfaces with piecewise geodesic boundary exist in the compactification of the moduli space when there are cone points with cone angles greater than $\pi$ along with boundary geodesics. This implies that the natural compactification of the moduli space of hyperbolic surfaces has real codimension one boundary components.  These are formed by  hyperbolic surfaces with piecewise geodesic boundary containing a single  cone angle.  The rotation of the boundary structure with respect to the rest of the surface gives the extra real dimension.  A piecewise geodesic boundary with more than one cone angle gives a lower dimensional stratum.

\section{Generalised Weil-Petersson symplectic form}   \label{WPsymp}

In this section we study the behaviour of the generalised Weil-Petersson forms $\wpl$ inside the compactification of the moduli space.

Given any family of hyperbolic surfaces $X\to S$ the relative canonical line bundle $K_{X/S}\to X$ comes equipped with a natural Hermitian metric defined via the unique hyperbolic metric on each fibre.  The holomorphic structure and Hermitian metric uniquely determines a connection on $K_{X/S}$ with curvature that represents a Chern form.  Wolpert used this idea in \cite{WolChe} to prove an explicit pushforward formula for the Weil-Petersson symplectic form in terms of the curvature, which appears as the $n=0$ case of \eqref{intfibre}.  

Schumacher and Trapani generalised the result of Wolpert by relating $\wpl$ with the curvature of a Hermitian metric on the relative log-canonical line bundle over the universal curve \cite{STrWei}.  The symplectic form $\wpl$ is a pushforward of the square of the curvature of the Hermitian metric, or Chern form, constructed as follows.  Given any family of hyperbolic surfaces $X\to S$ with $n$ sections having image divisor $D=\displaystyle\sqcup_{j=1}^n D_j\subset X$, fix a choice of weights $\mathbf{a}=(a_1,...,a_n)\in(0,1]^n$ and define $\mathbf{a}\cdot D=\sum a_jD_j$.  The relative log-canonical line bundle $K_{X/S}(D)$ is equipped with a Hermitian metric, using the conical hyperbolic metric $h(\mathbf{L})|dz|^2$ on fibres determined by $\mathbf{a}\cdot D$ where $\mathbf{a}=a(\mathbf{L})$.  The curvature of the Chern connection of the Hermitian metric, which represents the first Chern form of $K_{X/S}(D)$ adjusted by the weights $\mathbf{a}$, is given by the real $(1,1)$ form defined over $X\setminus D$ by
\begin{equation}   \label{11form}
\Omega_{X/S}(\mathbf{a}):=i\partial\bar{\partial}\log(h({\mathbf{L}})).
\end{equation}
Note that \eqref{11form} is well-defined since the ambiguity $h(\mathbf{L})\mapsto h(\mathbf{L})|z'(w)|^2$ due to a change of local coordinate $z=z(w)$ satisfies $\partial\bar{\partial}\log|z'(w)|^2=0$.
Schumacher and Trapani proved the following:
\begin{theorem}[\cite{STrWei}]  
\begin{equation}  \label{intfibre}
\frac12\int_{X/S}\Omega_{X/S}(\mathbf{a})^2=\wpl
\end{equation}
\end{theorem}
\noindent When $n=0$ this is Wolpert's pushforward formula \cite{WolChe} which was generalised to $n>0$ and $\mathbf{a}=(1,1,...,1)$ by Arbarello and Cornalba in \cite{ACoCom}.  Note that our conventions in \eqref{11form} and \eqref{intfibre} differ by constants from those in \cite{STrWei} in order to produce the usual convention $[\omega_{\mathrm{WP}}]=2\pi^2\kappa_1$ when $\mathbf{a}=(1,1,...,1)$.

\subsection{Cohomology class on the compactification}   \label{sec:cohclass}

Wolpert proved in \cite{WolWei} that the Weil-Petersson form $\omega_{\mathrm{WP}}$ extends from $\modm_{g,n}$ to a closed current on $\overline{\modm}_{g,n}$.  In \cite{WolWei}, the result is stated for $n=0$ but generalises to $n>0$ due to the local nature of the proof.  The following result generalises this to allow cone angles.

\begin{proposition}\label{comsumpi}
Given $\mathbf{L}\in i\hspace{.5mm}[0,2\pi)^n$ satisfying \eqref{areaineq}, 
the Weil-Petersson form $\wpl$ extends as a closed current to the Hassett compactification $\overline{\modm}_{g,\mathbf{a}}\cong\overline{\modm}_{g,n}^{\mathrm{hyp}}(\mathbf{L})$ 
for $\mathbf{a}=a(\mathbf{L})$.
\end{proposition}
\begin{proof}
The $n=0$ case was proven by Wolpert in \cite[Theorem 2.3]{WolWei}.  The extension requires knowledge of the asymptotic behaviour of components of the Petersson pairing \eqref{WPmetric} near cusps on a hyperbolic surface.  Wolpert achieves this by using estimates of Masur \cite{MasExt}, who studied the asymptotic behaviour of the Weil-Petersson metric near the boundary divisor of $\overline{\modm}_{g}$, to show convergence of the pairing between $\omega_{\mathrm{WP}}$ and locally compact differential forms.   Since Masur's results use the local structure of nodal curves near nodes, Wolpert's proof generalises to the $n>0$ case.   By Theorem~\ref{comphomeo} together with the behaviour of the universal curve proven in Section~\ref{sec:compact}, we see that the same local structure of nodal curves near nodes occurs for conical hyperbolic surfaces, i.e.\  in the boundary limit of the compactification $\overline{\modm}_{g,n}^{\mathrm{hyp}}(\mathbf{L})$ of $\modm_{g,n}$, nodes form and the hyperbolic metric is complete there, hence Wolpert's proof generalises to the case of cone angles.  In particular, the local behaviour of $\wpl$ near the boundary of the compactification is the same as the local behaviour of $\omega_{\mathrm{WP}}$ in the complete case.  Hence the pairing between $\wpl$ and locally compact differential forms converges as for the complete case, and Wolpert's argument generalises to show that the 2-form extends as a well-defined closed current.
\end{proof}

On the Deligne-Mumford compactification, define the forgetful map $\pi:\overline{\modm}_{g,n+1}\to\overline{\modm}_{g,n}$ that forgets the last point, and define  $K_{\overline{\modm}_{g,n+1}/\overline{\modm}_{g,n}}$ to be the relative canonical sheaf of the map $\pi$.  There are natural sections $p_j:\overline{\modm}_{g,n}\to\overline{\modm}_{g,n+1}$ of the map $\pi$ for each $j=1,...,n$.   Define $L_j= p_j^*K_{\overline{\modm}_{g,n+1}/\overline{\modm}_{g,n}}$ to be the line bundle $L_j\to\overline{\modm}_{g,n}$ with fibre above $[(C,p_1,\ldots,p_n)]$ given by $T_{p_j}^*C$ and $\psi_j=c_1(L_j)\in H^{2}(\overline{\modm}_{g,n},\bq) $ to be its first Chern class.   Define the classes $\kappa_m=\pi_*\psi^{m+1}_{n+1}\in H^{2m}( \overline{\cal M}_{g,n},\bq)$ where $\pi$ is the forgetful map $\overline{\modm}_{g,n+1}\stackrel{\pi}{\longrightarrow}\overline{\modm}_{g,n}$.  They were defined by Mumford \cite{MumTow} and in this form by Arbarello-Cornalba \cite{ACoCom}, Miller \cite{MilHom}, Morita \cite{MorCha} and Witten \cite{WitTwo}.  We will refer to $\psi_j$ and $\kappa_m$ as $\psi$ classes, respectively $\kappa$ classes.
 Wolpert proved that the cohomology class of the extension of the Weil-Petersson form is $[\omega_{\mathrm{WP}}]=2\pi^2\kappa_1$ \cite{WolWei}.  

The following special case of Theorem~\ref{extension} is required in the proof of Theorem~\ref{extension}.  It is equivalent to the calculation of the volume $\mathrm{Vol}(\modm_{0,4}^{\mathrm{hyp}}(\mathbf{L}))$ since the cohomology class is of top degree on $\overline{\modm}_{0,4}$.

\begin{proposition}   \label{04calculation}
If $\mathbf{L}\in i[0,2\pi)^4$ and $|L_j|+|L_k|<2\pi,\ \forall j\neq k$ then 
\[
[\wpl]=2\pi^2\kappa_1-\tfrac12\sum_{j=1}^4\theta_j^2\psi_j.
\]
\end{proposition}
\begin{proof}
The 2-form $\Omega_{X/S}(\mathbf{a})$, for $X=\modm_{0,5}$ and $S=\modm_{0,4}$ defined on the universal curve $\modm_{0,5}$ over $\modm_{0,4}$ is given explicitly in \eqref{11form}.  It extends to a closed current on $\overline{\modm}_{0,5}$ which represents a cohomology class $\eta\in H^2(\overline{\modm}_{0,5},\br)$ determined via the following direct calculation.  The moduli space $\overline{\modm}_{0,5}$ can be obtained as the blowup of $\bp^1\times\bp^1$ at three points as in the diagram which shows $\overline{\modm}_{0,5}\to\overline{\modm}_{0,4}$, where the $D_j$ are the marked points.
\begin{center}
 \begin{tikzpicture}[x=0.5pt,y=0.5pt,yscale=-1.07,xscale=1.06]

\draw    (230,230) -- (70,230) ;
\draw    (230,120) -- (70,120) ;
\draw    (230,170) -- (70,170) ;
\draw    (230,70) -- (70,70) ;
\draw    (90,180) -- (70,200) ;
\draw    (140,130) -- (110,160) ;
\draw    (190,80) -- (160,110) ;
\draw    (230,40) -- (210,60) ;
\draw  [draw opacity=0][dash pattern={on 4.5pt off 4.5pt}] (110.74,160.07) .. controls (115.83,165.58) and (115.28,174.4) .. (109.48,179.81) .. controls (103.67,185.22) and (94.8,185.15) .. (89.66,179.64) .. controls (89.65,179.63) and (89.64,179.62) .. (89.63,179.61) -- (100.18,169.83) -- cycle ; \draw  [dash pattern={on 4.5pt off 4.5pt}] (110.74,160.07) .. controls (115.83,165.58) and (115.28,174.4) .. (109.48,179.81) .. controls (103.67,185.22) and (94.8,185.15) .. (89.66,179.64) .. controls (89.65,179.63) and (89.64,179.62) .. (89.63,179.61) ;  
\draw  [draw opacity=0][dash pattern={on 4.5pt off 4.5pt}] (160.74,110.07) .. controls (165.83,115.58) and (165.2,124.48) .. (159.3,129.98) .. controls (153.39,135.49) and (144.44,135.49) .. (139.3,129.98) .. controls (139.22,129.89) and (139.14,129.79) .. (139.05,129.7) -- (150,120) -- cycle ; \draw  [dash pattern={on 4.5pt off 4.5pt}] (160.74,110.07) .. controls (165.83,115.58) and (165.2,124.48) .. (159.3,129.98) .. controls (153.39,135.49) and (144.44,135.49) .. (139.3,129.98) .. controls (139.22,129.89) and (139.14,129.79) .. (139.05,129.7) ;  
\draw  [draw opacity=0][dash pattern={on 4.5pt off 4.5pt}] (210.74,60.07) .. controls (215.83,65.58) and (215.2,74.48) .. (209.3,79.98) .. controls (203.39,85.49) and (194.44,85.49) .. (189.3,79.98) .. controls (189.22,79.89) and (189.14,79.79) .. (189.05,79.7) -- (200,70) -- cycle ; \draw  [dash pattern={on 4.5pt off 4.5pt}] (210.74,60.07) .. controls (215.83,65.58) and (215.2,74.48) .. (209.3,79.98) .. controls (203.39,85.49) and (194.44,85.49) .. (189.3,79.98) .. controls (189.22,79.89) and (189.14,79.79) .. (189.05,79.7) ;  
\draw  [draw opacity=0] (81.06,152.81) .. controls (87.83,154.57) and (94.9,160.44) .. (99.56,169) .. controls (104.05,177.24) and (105.24,186.03) .. (103.35,192.6) -- (83.79,177.59) -- cycle ; \draw   (81.06,152.81) .. controls (87.83,154.57) and (94.9,160.44) .. (99.56,169) .. controls (104.05,177.24) and (105.24,186.03) .. (103.35,192.6) ;  
\draw  [draw opacity=0] (131.06,102.81) .. controls (137.83,104.57) and (144.9,110.44) .. (149.56,119) .. controls (154.05,127.24) and (155.24,136.03) .. (153.35,142.6) -- (133.79,127.59) -- cycle ; \draw   (131.06,102.81) .. controls (137.83,104.57) and (144.9,110.44) .. (149.56,119) .. controls (154.05,127.24) and (155.24,136.03) .. (153.35,142.6) ;  
\draw  [draw opacity=0] (182.06,53.81) .. controls (188.83,55.57) and (195.9,61.44) .. (200.56,70) .. controls (205.05,78.24) and (206.24,87.03) .. (204.35,93.6) -- (184.79,78.59) -- cycle ; \draw   (182.06,53.81) .. controls (188.83,55.57) and (195.9,61.44) .. (200.56,70) .. controls (205.05,78.24) and (206.24,87.03) .. (204.35,93.6) ;  
\draw  [draw opacity=0] (205.89,82.4) .. controls (200.32,85.31) and (193.52,84.5) .. (189.3,79.98) .. controls (189.3,79.97) and (189.29,79.96) .. (189.28,79.95) -- (200,70) -- cycle ; \draw   (205.89,82.4) .. controls (200.32,85.31) and (193.52,84.5) .. (189.3,79.98) .. controls (189.3,79.97) and (189.29,79.96) .. (189.28,79.95) ;  
\draw  [draw opacity=0] (155.89,132.4) .. controls (150.32,135.31) and (143.52,134.5) .. (139.3,129.98) .. controls (139.3,129.97) and (139.29,129.96) .. (139.28,129.95) -- (150,120) -- cycle ; \draw   (155.89,132.4) .. controls (150.32,135.31) and (143.52,134.5) .. (139.3,129.98) .. controls (139.3,129.97) and (139.29,129.96) .. (139.28,129.95) ;  
\draw  [draw opacity=0] (106.07,182.23) .. controls (100.5,185.14) and (93.7,184.33) .. (89.48,179.81) .. controls (89.48,179.8) and (89.47,179.79) .. (89.46,179.78) -- (100.18,169.83) -- cycle ; \draw   (106.07,182.23) .. controls (100.5,185.14) and (93.7,184.33) .. (89.48,179.81) .. controls (89.48,179.8) and (89.47,179.79) .. (89.46,179.78) ;  
\draw  [fill={rgb, 255:red, 0; green, 0; blue, 0 }  ,fill opacity=1 ] (98,230) .. controls (98,228.9) and (98.9,228) .. (100,228) .. controls (101.1,228) and (102,228.9) .. (102,230) .. controls (102,231.1) and (101.1,232) .. (100,232) .. controls (98.9,232) and (98,231.1) .. (98,230) -- cycle ;
\draw  [fill={rgb, 255:red, 0; green, 0; blue, 0 }  ,fill opacity=1 ] (148,230) .. controls (148,228.9) and (148.9,228) .. (150,228) .. controls (151.1,228) and (152,228.9) .. (152,230) .. controls (152,231.1) and (151.1,232) .. (150,232) .. controls (148.9,232) and (148,231.1) .. (148,230) -- cycle ;
\draw  [fill={rgb, 255:red, 0; green, 0; blue, 0 }  ,fill opacity=1 ] (198,230) .. controls (198,228.9) and (198.9,228) .. (200,228) .. controls (201.1,228) and (202,228.9) .. (202,230) .. controls (202,231.1) and (201.1,232) .. (200,232) .. controls (198.9,232) and (198,231.1) .. (198,230) -- cycle ;
\draw  [fill={rgb, 255:red, 0; green, 0; blue, 0 }  ,fill opacity=1 ] (98,170) .. controls (98,168.9) and (98.9,168) .. (100,168) .. controls (101.1,168) and (102,168.9) .. (102,170) .. controls (102,171.1) and (101.1,172) .. (100,172) .. controls (98.9,172) and (98,171.1) .. (98,170) -- cycle ;
\draw  [fill={rgb, 255:red, 0; green, 0; blue, 0 }  ,fill opacity=1 ] (148,120) .. controls (148,118.9) and (148.9,118) .. (150,118) .. controls (151.1,118) and (152,118.9) .. (152,120) .. controls (152,121.1) and (151.1,122) .. (150,122) .. controls (148.9,122) and (148,121.1) .. (148,120) -- cycle ;
\draw  [fill={rgb, 255:red, 0; green, 0; blue, 0 }  ,fill opacity=1 ] (198,70) .. controls (198,68.9) and (198.9,68) .. (200,68) .. controls (201.1,68) and (202,68.9) .. (202,70) .. controls (202,71.1) and (201.1,72) .. (200,72) .. controls (198.9,72) and (198,71.1) .. (198,70) -- cycle ;
\draw  [fill={rgb, 255:red, 0; green, 0; blue, 0 }  ,fill opacity=1 ] (102,183) .. controls (102,181.9) and (102.9,181) .. (104,181) .. controls (105.1,181) and (106,181.9) .. (106,183) .. controls (106,184.1) and (105.1,185) .. (104,185) .. controls (102.9,185) and (102,184.1) .. (102,183) -- cycle ;
\draw  [fill={rgb, 255:red, 0; green, 0; blue, 0 }  ,fill opacity=1 ] (152,133) .. controls (152,131.9) and (152.9,131) .. (154,131) .. controls (155.1,131) and (156,131.9) .. (156,133) .. controls (156,134.1) and (155.1,135) .. (154,135) .. controls (152.9,135) and (152,134.1) .. (152,133) -- cycle ;
\draw  [fill={rgb, 255:red, 0; green, 0; blue, 0 }  ,fill opacity=1 ] (202.5,83) .. controls (202.5,81.9) and (203.4,81) .. (204.5,81) .. controls (205.6,81) and (206.5,81.9) .. (206.5,83) .. controls (206.5,84.1) and (205.6,85) .. (204.5,85) .. controls (203.4,85) and (202.5,84.1) .. (202.5,83) -- cycle ;

\draw (231,162.4) node [anchor=north west][inner sep=0.75pt]    {$D_{1}$};
\draw (231,62.4) node [anchor=north west][inner sep=0.75pt]    {$D_{3}$};
\draw (231,112.4) node [anchor=north west][inner sep=0.75pt]    {$D_{2}$};
\draw (221,22.4) node [anchor=north west][inner sep=0.75pt]    {$D_{4}$};
\draw (71,132.4) node [anchor=north west][inner sep=0.75pt]    {$E_{1}$};
\draw (171,32.4) node [anchor=north west][inner sep=0.75pt]    {$E_{3}$};
\draw (121,82.4) node [anchor=north west][inner sep=0.75pt]    {$E_{2}$};

\draw (150,190) node [anchor=north west][inner sep=0.75pt]    {$\downarrow$};
\draw (95,234.4) node [anchor=north west][inner sep=0.75pt]    {$0$};

\draw (146,235.4) node [anchor=north west][inner sep=0.75pt]    {$1$};

\draw (191,239.2) node [anchor=north west][inner sep=0.75pt]    {$\infty $};

\end{tikzpicture}
\end{center}
The second homology is $H_2(\overline{\modm}_{0,5},\bz)\cong\bz^5=\langle H,F,E_1,E_2,E_3\rangle$ where $F$ represents a generic fibre of the map $\overline{\modm}_{0,5}\to\overline{\modm}_{0,4}$, the exceptional divisors in the blowup are given by $E_j$, and $H$ represents the preimage of a generic curve $\bp^1\times\{\mathrm{pt}\}$ in the blowup of $\bp^1\times\bp^1$.  It satisfies $H\cdot F=1$, $H\cdot H=H\cdot E_j=0$ for $j=1,2,3$.  With respect to these generators,
$K_{X/S}=-2H+\sum E_j$, $D_j=H-E_j$ for $j=1,2,3$ and $D_4=F+H-\sum E_j$.

A cohomology class on $\overline{\modm}_{0,5}$ is determined by its evaluation on these homology classes.  This can be achieved by integration of $\Omega_{X/S}(\mathbf{a})$ over smooth curves 
as follows.
The component of the curvature 2-form $\Omega_{X/S}(\mathbf{a})=i\partial\bar{\partial}\log(h({\mathbf{L}}))$ along a fibre of the map $\overline{\modm}_{0,5}\to\overline{\modm}_{0,4}$ defines the hyperbolic area form since the Gaussian curvature is $K=\frac{-2}{h({\mathbf{L}})}\frac{\partial^2}{\partial z\partial\bar{z}}\log(h({\mathbf{L}}))$.  By the Gauss-Bonnet formula, the hyperbolic area is constant over fibres,  and given by
\[ \langle\eta,F\rangle=\mathrm{area}(F)=4\pi-\sum_{j=1}^4\theta_j.
\]
To evaluate $\eta$ on the classes $E_j$, we consider limits of fibres.
As a point approaches 0, 1 or $\infty$ in $\overline{\modm}_{0,4}$, the hyperbolic surface becomes nodal and its area localises, again by the Gauss-Bonnet formula, in a well-behaved way.  The fibre over $0\in\overline{\modm}_{0,4}$  is given by $E_1\cup F_1$ where $F_1$ is the proper transform of $\{0\}\times\bp^1$.  It is represented by the nodal hyperbolic surface with cone angles $\theta_1$, $\theta_4$ and a cusp on the component corresponding to $E_1$, and cone angles $\theta_2$, $\theta_3$ and a cusp on the component corresponding to $F_1$.  The areas of the components are $2\pi-\theta_1-\theta_4$ and $2\pi-\theta_2-\theta_3$.  Hence 
\[\langle\eta,E_1\rangle=2\pi-\theta_1-\theta_4.\]  
Note that $E_1+F_1=F$ in homology which is reflected by the fact that the sum of the areas of $E_1$ and $F_1$ equals the area of a general fibre $F$.  Similar arguments yield $\langle\eta,E_j\rangle$ for $j=2,3$.  

To evaluate $\eta$ on $H$, represent $H$ by the proper transform of $\bp^1\times\{z_0\}$ where $z_0\notin\{0,1,\infty\}$ so that $H\cap D_j=\emptyset$ for $j=1,2,3$ and $H\cap D_4=\{z_0\}$.  The family of hyperbolic metrics $h({\mathbf{L}})$ restricts to $H$ to give $h(z_0,w)|dz|^2$ where $h\sim|w-z_0|^{-2a_4}$ as $w\to z_0$.  The integral of $i\partial\bar{\partial}\log(h({\mathbf{L}}))$ along $H$ depends only on the smoothness of $h$ for $w\neq z_0$ and its asymptotic behaviour near $z_0$, since given any other function, say $\xi(w)h(w)$ for a smooth non-vanishing function $\xi(w)$ on $H\cong S^2$, we have 
\[ \int_{S^2}i\partial\bar{\partial}\log(\xi(w)h(w))=\int_{S^2}i\partial\bar{\partial}\log(\xi(w))+\int_{S^2}i\partial\bar{\partial}\log(h(w))=\int_{S^2}i\partial\bar{\partial}\log(h(w))
\]
where the second equality uses the vanishing of the integral of the closed form $i\partial\bar{\partial}\log(\xi(w))$.  Hence, to evaluate $\eta$ on $H$, we can choose any function $h(w)$ smooth for $w\neq z_0$ and satisfying $h\sim|w-z_0|^{-2a_4}$ as $w\to z_0$. By a change of coordinates, choose $z_0=0$ and let 
\[h(w)=|w|^{-2a_4\chi(|w|)}\]
where $\chi(r)$ is a smooth step function from 0 to 1, specifically $\chi(r)=\left\{\begin{array}{ll}1&r\leq 1\\
0&r\geq2\end{array}\right.$ and it is smooth on $r\in[1,2]$.  Now $2i\partial\bar{\partial}=d*d$ so
\[ \int_{S^2}\hspace{-1mm}i\partial\bar{\partial}\log(h(w))=\tfrac12\int_{S^2}\hspace{-1mm}d*d\log(h(w))=\tfrac12\int_{|w|=1/2}\hspace{-8mm}*d\log(|w|^{-2a_4})=-a_4\int_{|w|=1/2}\hspace{-8mm}*d\log(|w|)=2\pi a_4\]
where the last equality uses  $*d\log(|w|)=\frac12*d\log(w\bar{w})=\frac12*(\frac{dw}{w}+\frac{d\bar{w}}{\bar{w}})=\frac{i}{2}(\frac{dw}{w}-\frac{d\bar{w}}{\bar{w}})$ and $\int_{|w|=1/2}(\frac{dw}{w}-\frac{d\bar{w}}{\bar{w}})=4\pi i$.

Hence we have:
\[\langle\eta,F\rangle=4\pi-\sum_1^4\theta_j,\quad\langle\eta,H\rangle=2\pi-\theta_4,\quad\langle\eta,E_j\rangle=2\pi-\theta_j-\theta_4\quad j=1,2,3.
\]

The class $2\pi(K_{X/S}+\mathbf{a}\cdot D)$ has the same intersection properties with $H$, $F$ and $E_j$ as those above.  For example,  
\begin{align*}
2\pi(K_{X/S}+\mathbf{a}\cdot D)\cdot E_1&=2\pi (-2H+\sum E_j+\sum a_jD_j)\cdot E_1\\
&=2\pi(-1+a_1+a_4)=2\pi-\theta_1-\theta_4
\end{align*}
as claimed.  The intersection numbers uniquely determine $\eta$, hence
\[\eta=2\pi(K_{X/S}+\mathbf{a}\cdot D).
\]
Apply the Gysin homomorphism $\pi_*(\eta^2)$ to produce the integration along fibres from \eqref{intfibre}.  Then
\[\eta^2=4\pi^2(K_{X/S}+\mathbf{a}\cdot D)^2=4\pi^2(K_{X/S}^2+2(\mathbf{a}\cdot D)\cdot K_{X/S}+(\mathbf{a}\cdot D)^2)
\]
hence 
\begin{align*}
\tfrac12\pi_*(\eta^2)&=2\pi^2(\pi_*(K_{X/S}^2)+\sum_j (2a_j \psi_j-a_j^2\psi_j))\\
&=2\pi^2(\kappa_1-\sum_j(1-a_j)^2\psi_j)
\end{align*}
which uses $\pi_*(D_j\cdot D_j)=-\psi_j$, $D_j\cdot D_k=0$, $\pi_*(D_j\cdot K_{X/S})=\psi_j$ and  $\kappa_1=\pi_*(K_{X/S}^2)+\sum_1^4 \psi_j$.  
Put $\theta_j=2\pi(1-a_j)$ and write $[\wpl]:=\eta$ so that
\[
[\wpl]=2\pi^2\kappa_1-\tfrac12\sum_{j=1}^4\theta_j^2\psi_j
\]
as claimed.
\end{proof}

\subsubsection{Wolpert's formula}
Theorem~\ref{extension}, restated in the following proposition, generalises Wolpert's formula $[\omega^{\mathrm WP}]=2\pi^2\kappa_1$, \cite{WolChe}, for the cohomology class of the extension of the Weil-Petersson form to a closed current, and its generalisation to the $n>0$ case with $\mathbf{a}=(1,...,1)$ proven in \cite{ACoCom}.  

\begin{proposition}  \label{th:WPcoh}
If
$\mathbf{L}\in\left\{(i\theta_1,...,i\theta_n)\in i[0,2\pi)^n\mid \theta_j+\theta_k<2\pi,\ \forall j\neq k  \right\}$ 
then the extension of $\wpl$ defined in Proposition~\ref{comsumpi} as a closed current on $\overline{\modm}_{g,n}$ defines the cohomology class 
\[\left[\wpl\right]=2\pi^2\kappa_1-\tfrac12\sum_j\theta_j^2\psi_j\in H^2(\overline{\modm}_{g,n},\br).
\]
where $\wpl$ denotes both the Weil-Petersson 2-form and its extension as a current to the compactification. 
\end{proposition}
\begin{proof}

The proof follows the elegant inductive proof of Wolpert's formula given by Arbarello and Cornalba in \cite{ACoCom}.  The main idea is to show that each side of \eqref{WPcoh} restricts to the boundary strata given by smaller moduli spaces, 
\[ \overline{\modm}_{g-1,n+2}\to\overline{\modm}_{g,n},\qquad\overline{\modm}_{g_1,n_1+1}\times \overline{\modm}_{g_2,n_2+1}\to\overline{\modm}_{g,n}
\] 
via the same formula \eqref{WPcoh} applied to the smaller moduli spaces.  Equivalently the aim is to show that the difference between the two sides of \eqref{WPcoh} (inductively) vanishes on the boundary.

The symplectic form $\wpl$ restricts naturally to the boundary strata.   The cotangent space to a point $p\in\overline{\modm}_{g,n}$ that represents the nodal pointed curve $(C,D)$ consists of meromorphic quadratic differentials on $C$ with at worst simple poles at the marked points $D=(p_1,...,p_n)$ and at worst double poles at the nodes of $C$.  If $p$ lies in the image of $\rho:\overline{\modm}_{g-1,n+2}\to\overline{\modm}_{g,n}$, then the restriction of the Petersson pairing \eqref{WPmetric} to the cotangent space $T^*_p\overline{\modm}_{g-1,n+2}$ pairs meromorphic quadratic differentials which, rather than double poles, have at worst simple poles at the node $p_{n+1}\sim p_{n+2}$.  Hence $\rho^*\wpl=\omega_{\mathbf{a},1,1}$ since the pairings agree when weights of 1 are assigned to the nodes.  On the image of $\phi_I:\overline{\modm}_{g_1,|I|+1}\times \overline{\modm}_{g_2,|J|+1}\to\overline{\modm}_{g,n}$,  by the same reasoning as above, the pairing restricts to the pairing on each factor of the domain and sends covectors from different domains to zero, so $\phi_I^*\wpl=\omega_{\mathbf{a}_I,1}\otimes1+1\otimes\omega_{\mathbf{a}_J,1}$ where $\mathbf{a}=(\mathbf{a}_I,\mathbf{a}_J)$.

The class $2\pi^2\kappa_1-\tfrac12\sum_j\theta_j^2\psi_j$ on the right hand side of \eqref{WPcoh} also restricts naturally to each boundary divisor of $\overline{\modm}_{g,n}$.   The restrictions to $\overline{\modm}_{g-1,n+2}$ are given by $\rho^*\kappa_1=\kappa_1$ and $\rho^*\psi_j=\psi_j$ hence 
\[\rho^*(2\pi^2\kappa_1-\tfrac12\sum_j\theta_j^2\psi_j)=2\pi^2\kappa_1-\tfrac12\sum_j\theta_j^2\psi_j.
\]
Similarly, for $I\sqcup J=\{1,...,n\}$ the restrictions to $\overline{\modm}_{g_1,|I|+1}\times \overline{\modm}_{g_2,|J|+1}$ are given by $\phi_I^*\kappa_1=\kappa_1\otimes 1+1\otimes\kappa_1$ and $\phi_I^*\psi_j=\psi_j\otimes 1$ for $j\in I$, and $\phi_I^*\psi_j=1\otimes\psi_j$ for $j\notin I$.  Hence
\[\phi_I^*(2\pi^2\kappa_1-\tfrac12\sum_j\theta_j^2\psi_j)=(2\pi^2\kappa_1-\tfrac12\sum_{j\in I}\theta_j^2\psi_j)\otimes1+1\otimes(2\pi^2\kappa_1-\tfrac12\sum_{j\in J}\theta_j^2\psi_j).\]  

The restrictions of $\wpl$ and $2\pi^2\kappa_1-\tfrac12\sum_j\theta_j^2\psi_j$ to boundary divisors both take the same form.  This allows an inductive method to show that that they agree in cohomology.
Given the inductive assumption that $[\wpl]$ and $2\pi^2\kappa_1-\tfrac12\sum_j\theta_j^2\psi_j$ agree on $\overline{\modm}_{g',n'}$ for $2g'-2+n'<2g-2+n$, by restriction of these classes to the boundary divisors, as described above, the difference between the classes $[\wpl]$ and $2\pi^2\kappa_1-\tfrac12\sum_j\theta_j^2\psi_j$ vanishes on each boundary divisor of $\overline{\modm}_{g,n}$.     For moduli spaces of dimension greater than one, as used by Arbarello and Cornalba \cite{ACoCal, ACoCom}, this implies that the two classes coincide which uses the fact that $H^2(\overline{\modm}_{g,n},\br)$, and its restriction to each boundary divisor, is known explicitly---it is generated by $\kappa_1$, all $\psi_j$ and all boundary classes.  

The initial step in the induction argument requires verifying \eqref{WPcoh} on $\overline{\modm}_{1,1}$ and $\overline{\modm}_{0,4}$.  A degree two cohomology class in this dimension is equivalent to its evaluation, hence equality of the volumes $\mathrm{Vol}(\modm_{1,1}^{\mathrm{hyp}}(L))$, respectively $\mathrm{Vol}(\modm_{0,4}^{\mathrm{hyp}}(\mathbf{L}))$, with Mirzakhani's polynomial $V_{1,1}(L)$, respectively $V_{0,4}(\mathbf{L})$, is enough to prove the cohomological equality  of \eqref{WPcoh} in these cases.  The $(1,1)$ volume is calculated in Appendix~\ref{vol11} where it is proven that 
\[\mathrm{Vol}(\modm_{1,1}^{\mathrm{hyp}}(L))=V_{1,1}(L)\] 
as required.  The $(0,4)$ volume is calculated in Proposition~\ref{04calculation} where it is shown to also coincide with Mirzakhani's polynomial.  The induction argument is complete and the proposition is proven.
\end{proof}

\begin{remark}
In the proof of Proposition~\ref{04calculation} it is shown that the extension of the curvature form $\Omega_{X/S}(\mathbf{a})$ to a closed current when $X=\modm_{0,5}$ and $S=\modm_{0,4}$ defines the cohomology class 
\begin{equation}   \label{curvcoh}
[\Omega_{X/S}(\mathbf{a})]=2\pi(K_{X/S}+\mathbf{a}\cdot D).
\end{equation} 
For $\mathbf{a}=a(\mathbf{L})$, the proof generalises to give the formula \eqref{curvcoh}, now defined on the Hassett space $\overline{\modm}_{0,\mathbf{a}}$, for any $\mathbf{L}\in i[0,2\pi)^4$ which satisfies $\left|\sum_{j=1}^4L_j\right|<4\pi$, or equivalently $\sum_{j=1}^4a_j>2$.  The induction argument in the proof of Proposition~\ref{th:WPcoh} now applies to any $\mathbf{a}$ satisfying $|\mathbf{a}|>2-2g$, since the initial $(1,1)$ case remains the same and the initial $(0,4)$ case uses the generalisation.  This leads to a proof that the formula \eqref{curvcoh} holds in general on any Hassett space $\overline{\modm}_{g,\mathbf{a}}$ with a consequence that the volumes are polynomials locally in $\mathbf{L}$.  The polynomials change as one crosses walls in the space of Hassett stability conditions.  The calculations of these polynomials appears in \cite{AMNWei}.  Theorem~\ref{th:vol} shows that the polynomials agree with Mirzakhani's polynomials $V_{g,n}(\mathbf{L})$ in the main chamber of the space of stability conditions.  The polynomials in other chambers are different and in particular have different top degree terms.  In \cite{DNoCou,NorCou} polynomials with the same highest degree terms as $V_{g,n}(\mathbf{L})$ are produced via counting lattice points in $\modm_{g,n}$, and in the compactification $\overline{\modm}_{g,n}$.  It would be interesting to uncover analogous polynomials with the same highest degree terms for polynomials arising from different chambers in the space of stability conditions using different compactifications of $\modm_{g,n}$.
\end{remark}

Corollary~\ref{th:vol} is an immediate consequence of Proposition~\ref{th:WPcoh} since $\mathrm{Vol}(\modm_{g,n}^{\mathrm{hyp}}(\mathbf{L}))$ is given by $\int_{\overline{\modm}_{g,n}}\exp\big(2\pi^2\kappa_1+\sum_1^n\tfrac12L_j^2\psi_j\big)$ which agrees with Mirzakhani's polynomial.  

Corollary~\ref{old} now follows from \cite{DNoWei} but it can also be proven directly.  The vanishing
\[\lim_{\theta\to 2\pi}\mathrm{Vol}(\modm_{g,n+1}^{\mathrm{hyp}}(0^n,\theta))=0
\]
follows from the degeneration of the K\"ahler metric in the limit, proven in \cite{STrWei}.  One can also see it from Mondello's formula \cite{MonPoi} for the Poisson structure.  This argument shows that the vanishing in the $2\pi$ limit holds more generally, for example when all boundary components are cusps and cone angles, even when the volume is not given by Mirzakhani's polynomial.  

The derivative formula can be proven directly as follows.  By \eqref{WPcoh} when $\mathbf{L}=(0^n,\theta)$ then $\omega^{\mathrm WP}(\mathbf{L})=2\pi^2\kappa_1-\tfrac12\theta^2\psi_{n+1}$
and 
$ \mathrm{Vol}(\modm_{g,n+1}^{\mathrm{hyp}}(0^n,i\theta))=\int_{\overline{\modm}_{g,n+1}}\frac{\big(2\pi^2\kappa_1-\tfrac12\theta^2\psi_{n+1}\big)^N}{N!}$
where $N=3g-2+n$.  Hence
\begin{align*}
\left.\frac{\partial}{\partial \theta}\mathrm{Vol}(\modm_{g,n+1}^{\mathrm{hyp}}(0^n,i\theta))\right|_{\theta=2\pi}
&=\left.- \int_{\overline{\modm}_{g,n+1}}\theta\psi_{n+1}\frac{\big(2\pi^2\kappa_1-\tfrac12\theta^2\psi_{n+1}\big)^{N-1}}{(N-1)!}\right|_{\theta=2\pi}\\
&=-2\pi \int_{\overline{\modm}_{g,n+1}}\psi_{n+1}\frac{\big(2\pi^2(\kappa_1-\psi_{n+1})\big)^{N-1}}{(N-1)!}\\
&=2\pi(2-2g-n) \int_{\overline{\modm}_{g,n}}\frac{\big(2\pi^2\kappa_1\big)^{N-1}}{(N-1)!}\\
&=2\pi(2-2g-n)\mathrm{Vol}(\modm_{g,n}^{\mathrm{hyp}}(0^n))
\end{align*}
where the third equality uses the pullback formula $\kappa_1-\psi_{n+1}=\pi^*\kappa_1$ and the pushforward formula $\pi_*(\psi_{n+1}\pi^*\eta)=(2g-2+n)\eta$ for any $\eta\in H^*(\overline{\modm}_{g,n})$.

The volume $\mathrm{Vol}(\modm_{g,n}^{\mathrm{hyp}}(\mathbf{L}))$ for general $\mathbf{L}\in\br_{\geq0}^m\cup i\hspace{.5mm} \br_{>0}^n$ is not expected to be given by evaluation of Mirzakhani's polynomial because the Deligne-Mumford compactification is not in general related to a compactification of $\modm_{g,n}^{\mathrm{hyp}}(\mathbf{L})$.  When there exists boundary geodesics of positive length together with cone points of cone angle greater than $\pi$, the results of Section~\ref{sec:compact} show that the boundary of a compactification, which consists of surfaces in which a cone point meets a boundary geodesic, necessarily has real codimension one, which is different to the Deligne-Mumford compactification.

To prove that indeed evaluation of Mirzakhani's polynomial does not give the volume, we constructed an example $\modm_{0,4}^{\mathrm{hyp}}(0,0,\theta i,(2\pi-\epsilon)i)$ of a non-empty moduli space in the introduction  so that $V_{0,4}(0,0,\theta i,(2\pi-\epsilon)i)<0$ for $4\pi\epsilon<\theta^2$.  The moduli space is non-empty, hence has positive volume, since $\theta i+2\pi-\epsilon<4\pi$ so existence of conical metrics follows from \cite{McOPoi}.

\appendix

\section{Volume of \texorpdfstring{$\cM^{\mathrm{hyp}}_{1,1}(i\theta)$}{M}}  \label{vol11}
In this appendix we calculate the volume of $\cM^{\mathrm{hyp}}_{1,1}(i\theta)$ which is required in the proof of Theorem~\ref{extension}.

The main idea is to extend the method used by Wolpert \cite{WolKah} in his proof of the cuspidal case.  A derivation of this formula for rational cone angles was given in \cite{NNAWei}.

Let $M$ be a compact, oriented surface of genus $1$, with one boundary component. The fundamental group $\pi_1(M)$ is the free group on two generators, $\langle a,b \rangle$. A flat $\SL (2,\bbR)$-connection on $M$ is equivalent to a representative in the equivalence class of $\SL (2,\bbR)$ representations $\rho\in\Hom(\pi_1(M),\SL (2,\bbR))/\sim$, where the equivalence relation is given by conjugation. The boundary of $M$ is represented by the class $c=[a,b]\in\pi_1(M)$, and this corresponds to a cone point when $\tr(\rho(c))\in (-2,2)$. Such a representation gives $M$ the structure of a hyperbolic surface with cone angle  $\theta$ such that $-2\cos\left(\frac{\theta}{2}\right)=\tr \rho(c)$. \\ \indent 
There is an action of $\Aut(\pi_1(M))$ on $\Hom(\pi_1(M),\SL (2,\bbR))/\sim$, given by $g\cdot [\rho]=[\rho\circ g^{-1}]$. For $M$ a surface of type $(1,1)$, the automorphism group has three generators $g_1,\,g_2$ and $g_3$. The action of generators is given by, \begin{equation*}\begin{array}{ccc}
    g_1(a)=a^{-1} &  g_2(a)=b &  g_3(a)=ab\\
    g_1(b)=b      &  g_2(b)=a &  g_3(b)=b.
\end{array}\end{equation*}
One obtains a representation of $\Aut(\pi_1(M))$ into $\GL(2,\bbZ)$ by its action on $\bbZ^2$, generated by cosets of $a$ and $b$. Define $\Aut^+(\pi_1(M))\subset\Aut(\pi_1(M))$ as the preimage of $\SL(2,\bbZ)\subset\GL(2,\bbZ)$. This fits into a pullback square \begin{center}
\begin{tikzcd}[column sep = 4em, row sep = 4em]
\Aut^+(\pi_1(M)) \arrow[r,"\subset"] \arrow[d]  & \Aut(\pi_1(M)) \arrow[d] \\ 
\SL(2,\bbZ)  \arrow[r,"\subset"]  &  \GL(2,\bbZ).
\end{tikzcd}
\end{center}
Further, let $\Inn(\pi_1(M))\trianglelefteq\Aut^+(\pi_1(M))$ be the normal subgroup of inner automorphisms, with elements given by conjugation. The action of $\Inn(\pi_1(M))$ on $\Hom(\pi_1(M),\SL (2,\bbR))/\sim$ is trivial, so there is a well defined action of the quotient group $\Out^+(\pi_1(M)):=$ \newline $\Aut^+(\pi_1(M))/\Inn(\pi_1(M))$. \\ \indent
Define the character mapping \begin{align*}
    \chi:\Hom(\pi_1(M),\SL (2,\bbR))/\sim &\longrightarrow\bbR^3 \\
    \left[\rho\right]&\longmapsto \left[ \begin{array}{c}
         x(\rho)  \\
         y(\rho)  \\
         z(\rho)  \\
    \end{array} \right] :=\left[ \begin{array}{c}
         \tr(\rho(a))  \\
         \tr(\rho(b))  \\
         \tr(\rho(ab))  \\
    \end{array} \right].
\end{align*}
This is a homeomorphism \cite{FriThe,FKlVor}.  Level sets of the trace function \begin{equation}
    \kappa(x,y,z):=\tr(\rho [a,b])=x^2+y^2+z^2-xyz-2 
\end{equation}
describe regions in $\bbR^3$ corresponding to prescribed boundary holonomy. It is proven in \cite{GolMod} that any fixed $\kappa\in(-2,2)$ defines a surface in $\bbR^3$ with $5$ disconnected components. Further restriction to the region $x,y,z\geq2$ gives a copy of Teichm\"uller space $\cT^{\mathrm{hyp}}(\theta)$, where $\kappa=-2\cos\left(\frac{\theta}{2}\right)$. \\ \indent 
The mapping class group of $\cT^{\mathrm{hyp}}(\theta)$ consists of isotopy classes of homeomorphisms of surfaces. In the character variety, these correspond to automorphisms which preserve the polynomial $\kappa(x,y,z)$. A theorem of Dehn-Nielsen (proven in \cite{StiDeh}) says that this corresponds to the action of $\Out^+(\pi_1(M))\cong\SL(2,\bbZ)$ on the image of Teichm\"uller space in the character variety $\cT_{1,1}^\kappa=\left\{(x,y,z)\in\bbR^3|x^2+y^2+z^2-xyz-2=\kappa,\, x\geq 2,y\geq 2,z\geq2\right\}$. \\ \indent 
It will be convenient to work in the homogeneous coordinates $r=\frac{x}{yz},\,s=\frac{y}{zx}$ and $t=\frac{z}{xy}$. Here for a fixed $\kappa\in(-2,2)$ the Teichm\"uller space becomes $$\cT_{1,1}^\kappa=\left\{(r,s,t)\in\bbR^3\left|\right.r+s+t-1=(\kappa+2)rst,\, r>0,s>0,t>0\right\}.$$ 
Following \cite{WolKah}, define the subgroups 
\begin{itemize}
\item $\Gamma(2)=\left\{A\in\SL(2,\bbZ)\left|\right.A\equiv \left[\begin{matrix} 1 & 0 \\ 0 & 1\end{matrix}\right]\mod 2\right\}$, \\
\item $P\Gamma(2)$ as the image of $\Gamma(2)$ in $\PSL(2,\bbZ)$,
\item $\cM_2=\langle \phi_1,\,\phi_2,\,\phi_3\rangle\subset \GL(2,\bbZ)$, where $\phi_1=\left[\begin{matrix} -1 & -2 \\ 0 & 1\end{matrix}\right],\,\phi_2=\left[\begin{matrix} -1 & 0 \\ 0 & 1\end{matrix}\right]$ and $\phi_3=\left[\begin{matrix} 1 & 0 \\ -2 & 1\end{matrix}\right]$,
\item $\cM_2^+=\cM_2\cap\SL(2,\bbZ)$, corresponding to words of even length in generators of $\cM_2$. 
\end{itemize}
From \cite{WolKah} we have the following results; $[\PSL (2,\mathbb{Z}):\mathcal{M}_2^+]=6$, $[\cM_2:\mathcal{M}_2^+]=2$, and the action of $\cM_2$ on $\cT_{1,1}^\kappa$ is given by, 
\begin{equation*}\begin{array}{ccc}
    \phi_1(x)=yz-x &  \phi_2(x)=x&  \phi_3(x)=x \\
    \phi_1(y)=y      &  \phi_2(y)=y  &  \phi_3(y)=xz-y \\
    \phi_1(z)=z      &  \phi_2(z)=xy-z &  \phi_3(z)=z,
\end{array}\end{equation*} 
in $(x,y,z)$-coordinates and
\begin{equation*}\begin{array}{ccc}
    \phi_1(r)=1-r &  \phi_2(r)=\frac{tr}{1-t}&  \phi_3(r)=\frac{sr}{1-s} \\
    \phi_1(s)=\frac{rs}{1-r}      &  \phi_2(s)=\frac{ts}{1-t}  &  \phi_3(s)=1-s  \\
    \phi_1(t)=\frac{rt}{1-r}       &  \phi_2(t)=1-t &  \phi_3(t)=\frac{st}{1-s},
\end{array}\end{equation*}
in $(r,s,t)$-coordinates. \\ \indent
The main original part of the proof in this section is in computing the fundamental domain $\Delta$ for the $\cM_2$ action on $\cT_{1,1}^\kappa$ with $\kappa\in(-2,2)$. Wolpert does this in the case $\kappa=-2$ corresponding to cusped hyperbolic surfaces, and here the idea of proof will be similar. \vspace{0.2cm}
\begin{claim}
The fundamental domain $\Delta$ for the $\cM_2$ action on $\cT_{1,1}^\kappa$ with $\kappa\in(-2,2)$ is given by $\Delta=\{(r,s,t)\in\bbR^3\left|\right.r+s+t-1=(\kappa+2)rst,\,r\in\big(0,\frac{1}{2}\big],\,s\in\big(0,\frac{1}{2}\big],\,t\in\big(0,\frac{1}{2}\big]\}$. 
\end{claim}
\begin{proof of claim}
Consider the function $E(x,y,z)=x+y+z$. We will show that for any $X\in\cT_{1,1}^\kappa$, $E$ admits a unique minimum in $\Delta$, on the orbit $\cM_2(X)$. First note that $x,y,x\geq 2$ so $E(x,y,z)>0$. Also $\cM_2$ has discrete orbits, so a minimum for $E$ on $\cM_2(X)$ exists. \\ \indent
We show that minimums of $E$ on the orbit exist in $\Delta$. Consider the formulas, \begin{equation} \label{E equations}
\frac{E\circ\phi_1-E}{2yz}=\frac{1}{2}-\frac{x}{yz}, \qquad \frac{E\circ\phi_2-E}{2xy}=\frac{1}{2}-\frac{z}{xy}, \qquad \frac{E\circ\phi_3-E}{2zx}=\frac{1}{2}-\frac{y}{zx}.
\end{equation}
If $Y\in\cM_2(X)$ is a minimum for $E$, then $E\circ\phi_i(Y)\geq E(Y)$ for any $i=1,2$ or $3$. So by \eqref{E equations}, coordinates of $Y$ are such that $a,b,c\leq\frac{1}{2}$. So $E$ has a minimum in $\Delta$, on the orbit $\cM_2(X)$.  \\ \indent 
In what remains, we prove uniqueness of such a minimum. First, we prove a convexivity property for the function $E$. Let $w_{n+1}w_nw_{n-1}\cdots w_1$ be a reduced word of length $n+1$ in $\{\phi_1,\phi_2,\phi_3\}$, and let $W=w_nw_{n-1}\cdots w_1$. Note that $\phi_i^2=1$ for $i=1,2,3$, so in particular, $w_nW=w_{n-1}\cdots w_1$. For any $S\in\cT_{1,1}^\kappa$ such that $E(W(S))\geq E(w_{n-1}\cdots w_1(S))=E(w_{n}W(S))$, it is not possible that $E(W(S))\geq E(w_{n+1}W(S))$. \\ \indent
If it were possible, then applying equations \eqref{E equations} to $W(S)$ would imply that at least two of $r,s,t$ would be greater than or equal to $\frac{1}{2}$. Suppose with out loss of generality that $r\geq\frac12$ and $s\geq\frac12$, then $z^2=\frac{1}{rs}\leq 4$ so $z\in [-2,2]$. By definition of $W(S)\in\cT_{1,1}^\kappa$, this would mean that $z=2$, so $x^2+y^2+2^2-2xy=\kappa+2$, giving that $(x-y)^2=\kappa-2<0$. This is a contradiction, so the convexity property is established. \\ \indent 
Now, suppose $S_1, S_2\in\Delta$ are points in the orbit $\cM_2(X)$, which both reach a minimum value for $E$. Since $S_1, S_2\in\Delta$ \eqref{E equations} tells us that, $E(\phi_i(S_1))\geq E(S_1)$ and $E(\phi_i(S_2))\geq E(S_1)$ for $i=1,2,3$. Convexivity further implies that $E(\phi_j\phi_i(S_1))>E(\phi_iS_1)$ and $E(\phi_j\phi_i(S_2)) > E(\phi_i S_1)$ for $i=1,2,3$. Repeated application of convexivity and the strictness of these inequalities leave only two possibilities. Either $S_1=S_2$ or there exists a $k=1,2,3$, such that $S_1=\phi_k(S_2)$ (which is equivalent to $S_2=\phi_k(S_1)$). In the second case, $E(S_1)=E(\phi_k(S_2))= E(S_2)=E(\phi_k(S_1))$. By the definition of $E$ and $\phi$, it follows that $S_1=S_2$. This completes the proof.  
$ \strut\hfill\square$
\end{proof of claim} \vspace{-0.2cm}
Rewriting this region as depending on only two variables, we obtain $$\Delta=\left\{\left(r,s,\frac{r+s-1}{(\kappa+2)rs-2}\right)\in\bbR^3\left|\right.\,r\in\Big(\frac{1-2s}{2-(\kappa+2)s},\frac{1}{2}\Big],\,s\in\Big(0,\frac{1}{2}\Big]\right\}.$$
\begin{remark}
When $\theta\rightarrow 2\pi$ the fundamental domain contracts to a point. This explains why the volume vanishes in this case. 
\end{remark}
The Weil-Petersson form can be described explicitly on the character variety in $(x,y)$-coordinates, \begin{equation}
\omega=\frac{4dx\wedge dy}{xy-2z}.
\end{equation}
In $(r,s)$-coordinates this becomes \begin{equation}
\omega=\frac{dr\wedge ds}{(1-(\kappa+2)rs)rst}=\frac{dr\wedge ds}{(1-r-s)rs}. \end{equation}
Finally, the volume is given by \begin{align*}
\int_{\mathcal{M}_{1,1}(i\theta)}\omega &=\frac{1}{[\PSL (2,\mathbb{Z}):\mathcal{M}_2^+ ]}\int_{\mathcal{T}_{1,1}^\kappa/\mathcal{M}_2^+}\omega \\
&= \frac{[\cM_2:\mathcal{M}_2^+]}{[\PSL (2,\mathbb{Z}):\mathcal{M}_2^+]}\int_{\mathcal{T}_{1,1}^\kappa/\mathcal{M}_2^+}\omega \\ 
&= \frac{2}{6}\int_{\Delta} \frac{dr\wedge ds}{(1-r-s)rs}  \\
&= \frac{1}{3}\int_0^{\frac{1}{2}}\int_{\frac{1-2s}{2-(\kappa+2)s}}^{\frac{1}{2}}\frac{1}{(1-r-s)rs}dr\,ds\\
&= \frac16 \left(\pi^2-\theta^2/4 \right)& \theta\in [0,2\pi),
\end{align*}
where $\kappa=-2\cos\left(\frac{\theta}{2}\right)$.  A further factor of $1/2$ is needed due to the $\bz_2$ automorphism of the generic hyperbolic surface, yielding
\[\mathrm{Vol}(\modm_{1,1}^{\mathrm{hyp}}(i\theta))=\frac{1}{48}(4\pi^2-\theta^2)
\]
which agrees with Mirzakhani's polynomial evaluated at $i\theta$.

\section{Non-existence of pants decompositions}  \label{nopop}

In this appendix, we give a construction of arbitrary genus hyperbolic surface without a pants decomposition in a chosen isotopy class.
The two hexagons in Figure~\ref{conehex} naturally glue along the edges of given lengths in the case $w=x$ to produce a hyperbolic pair-of-pants with boundary lengths $X(x,y,z)$, $Y(x,y,z)$ and $Z(x,y,z)$ (respectively opposite the arcs of lengths $x$, $y$ and $z$).
\begin{figure}[ht]  
	\centerline{\includegraphics[height=8cm]{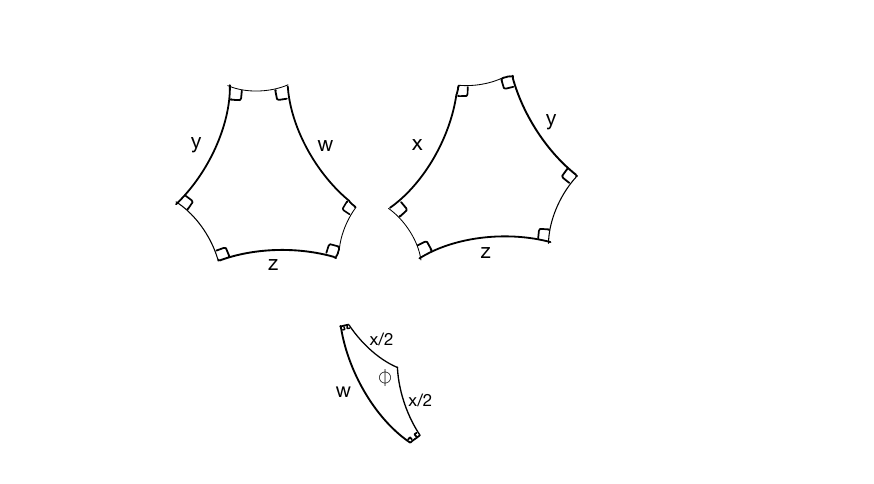}}
	\caption{Glue to produce surface with cone angle $\pi+\phi$}
	\label{fig:fatloop} \label{conehex}
\end{figure} 
The pentagon in Figure~\ref{conehex} has area given by $\pi-\phi$ and edge lengths satisfying
\[ \cosh(w)=-\cosh^2(x/2)\cos\phi+\sinh^2(x/2).
\]
When $\phi\approx\pi$ (and $\phi<\pi$), we have the area $\approx 0$ and $w\approx x$, due to $\cosh(w)=\cosh^2(w/2)+\sinh^2(w/2)$, so the pentagon is thin.  Glue the pentagon and hexagon along the edges of lengths $x/2$ and $x$ to create a cone angle of $\pi+\phi$.  Now glue the remaining matching sides of lengths $w$, $y$ and $z$ to obtain a genus zero hyperbolic surface $\Sigma$ with four boundary components---the cone point, and three geodesics of lengths $X$, $Y'>Y$ and $Z'>Z$. 
\begin{lemma}  \label{noPD}
The genus zero hyperbolic surface $\Sigma$ constructed from gluing the polygons in Figure~\ref{conehex} does not admit a geometric
pants decomposition.
\end{lemma}
\begin{proof}
Choose $\phi\approx\pi$ and consider a pants decomposition of $\Sigma$ defined by the isotopy class of a simple closed loop that separates $\Sigma$ into two pairs-of-pants consisting of the pair-of-pants $P$ containing the cone point and the boundary geodesic of length $Z'$ and its complement that contains the boundaries of lengths $X$ and $Y'$.  Then $\mathrm{area}(P)=\pi-\phi\approx 0$.  
Cut along the unique non-trivial closed geodesic from the cone point to itself to decompose $P$ into two annuli including $A\subset P$ that contains the boundary geodesic of length $Z'$.  Then $A$ has large area because $Z'$ and $x/2$ are large.  
The large area contradicts $\mathrm{area}(P)=\pi-\phi\approx 0$.  Thus $\Sigma$ does not admit a a pants decomposition in the given isotopy class.  

More generally, choose any topological pants decomposition of $\Sigma$ and suppose it admits a geometric pants decomposition in the same isotopy class.  One of the pair-of-pants $P$ contains the cone point.  Then the distance from the cone point to the boundary component of $\Sigma$ in $P$ is large and the argument above again produces a large area for an annulus inside $P$ hence a contradiction and the  proposition is proven.
\end{proof} 
\begin{corollary}  \label{noPDg}
There exist hyperbolic surfaces of arbitrary genus and arbitrary geometric boundary types, for example all cusps, except for one cone angle close to $2\pi$, without a pants decomposition in a chosen isotopy class.
\end{corollary}
\begin{proof}
Simply glue hyperbolic surfaces to the example $\Sigma$ from Lemma~\ref{noPD} and choose the pants decomposition isotopy class to contain the three geodesic boundary components of $\Sigma$.  For example, consider a genus 2 surface with a single cone angle and no other boundary components obtained by arranging $Y=Z$, hence $Y'=Z'$, and gluing along these two boundary components.  Then attach a genus 1 surface with geodesic boundary of length $X$ to produce a genus 2 hyperbolic surface with one cone point equipped with a topological pants decomposition not realisable geometrically.
\end{proof}


\begin{thebibliography}{100}
\bibitem{ACoCal} Arbarello,\ E. and Cornalba, M.
\emph{Calculating cohomology groups of moduli spaces of curves via algebraic geometry.}
Publ.\ Math.\ Inst.\ Hautes Etudes Sci.\ {\bf 88} (1998), 97-127.
 
\bibitem{ACoCom} Arbarello, E. and Cornalba, M.
\emph{Combinatorial and algebro-geometric cohomology classes on the moduli spaces of curves.}
 J. Algebraic Geom. {\bf 5} (1996), 705-749.
 
 \bibitem{AMNWei} Anagnostou, L., Mullane, S., and Norbury, P.
\emph{Weil-Petersson volumes, stability conditions and wall-crossing.}
\href{https://arxiv.org/abs/2310.13281}{arXiv:2310.13281}. 

\bibitem{AScGeo} Axelsson, R. and Schumacher, G.
\emph{Geometric approach to the Weil-Petersson symplectic form.}
Comment. Math. Helv.  {\bf 85}, (2010) 243-257.

\bibitem{BerExt} Bers, L. 
\emph{An extremal problem for quasiconformal mappings and a theorem by Thurston.} 
Acta Math. {\bf 141} (1978), 73-98. 

\bibitem{BerSpa} Bers, L. 
\emph{Spaces of degenerating Riemann surfaces.} In
Discontinuous groups and Riemann surfaces (Proc. Conf., Univ. Maryland, College Park, Md., 1973), p43-55. Ann. of Math. Studies, No. 79. Princeton Univ. Press, Princeton, N.J., 1974.

 \bibitem{BusCol} 
Buser, P.
\emph{The collar theorem and examples.}
Manuscripta Math {\bf 25} (1978), 349-357.

 \bibitem{BusGeo} 
Buser, P.
\emph{Geometry and spectra of compact Riemann surfaces.}
Springer Science \& Business Media (2010). 

 \bibitem{DPaCol} 
Dryden, E. and Parlier, H.
\emph{Collars and partitions of hyperbolic cone-surfaces.}
Geometriae Dedicata {\bf 127.1} (2007), 139-149, Springer.

\bibitem{DNoCou}
Do, N.  and Norbury, P.
\emph{Counting lattice points in compactified moduli spaces of curves.}
Geometry \& Topology {\bf 15} (2011), 2321-2350.

\bibitem{DNoWei} Do, N. and Norbury, P.
\emph{Weil-Petersson volumes and cone surfaces.}  
Geom. Dedicata {\bf 141} (2009), 93-107.

\bibitem{DuSim} Du, Y.
\emph{A Simple Recursion for the Mirzakhani Volume and its Super Extension.}
\href{http://arxiv.org/abs/2008.04458}{arXiv:2008.04458}

\bibitem{FMaPri}
Farb, B. and Margalit, D. 
\emph{A Primer on Mapping Class Groups.} 
Princeton Mathematical Series {\bf 49}. Princeton University Press, Princeton, NJ, (2012).

\bibitem{FriThe} Fricke, R.
\emph{{\"U}ber die Theorie der automorphen Modulgruppen.}
Nachrichten von der Gesellschaft der Wissenschaften zu G{\"o}ttingen, Mathematisch-Physikalische Klasse (1896), 91-101.

\bibitem{FKlVor} Fricke, R. and Klein, F.
\emph{Vorlesungen der Automorphen Funktionen..}
Leipzig. Vol. I (1897), Vol. II (1912).

\bibitem{GolInv} Goldman, W.
\emph{Invariant functions on Lie groups and Hamiltonian flows of surface group representations.}
Invent. Math.  {\bf 85} (1986), 263-302.

\bibitem{GolMod} Goldman, W.
\emph{The modular group action on real $SL (2)-$characters of a one-holed torus.}
Geometry \& Topology (I)  {\bf 7} (2003), 443-486, Mathematical Sciences Publishers.

\bibitem{HarCha} Harvey, W.
\emph{Chabauty spaces of discrete groups.}
Discontinuous groups and Riemann surfaces (Proc. Conf., Univ. Maryland, College Park, Md., 1973), p239-246. Ann. of Math. Studies, No. 79. Princeton Univ. Press, Princeton, N.J., 1974.

 \bibitem{HasMod} Hassett, B.
\emph{Moduli space of weighted pointed stable curves.} 
Adv. Math. {\bf 173} (2003), 316-352.

\bibitem{HKoAna} Hubbard, J. and Koch, S.
\emph{An analytic construction of the Deligne-Mumford compactification of the moduli space of curves.} 
J. Diff. Geom. {\bf 98} (2014), 261-313.

\bibitem{IzmSim} Izmestiev, I.
\emph{A simple proof of an isoperimetric inequality for Euclidean and hyperbolic cone-surfaces.}
Differ. Geom. Appl. {\bf 43} (2015), 95-101.

\bibitem{KeeCol} Keen, L.
\emph{Collars on Riemann surfaces.}
Discontinuous groups and Riemann surfaces (Proc. Conf., Univ. Maryland, College Park, Md., 1973). (1974), 263-268.

\bibitem{McOPoi} McOwen, C.
\emph{Point Singularities and Conformal Metrics on Riemann Surfaces.}
Proc.  AMS {\bf103} (1988), 222-224.

\bibitem{McSLar} McShane, G.
\emph{Large cone angles on a punctured sphere.}
\href{http://arxiv.org/abs/1503.00539}{arXiv:1503.00539}

\bibitem{MasExt} Masur, H.
\emph{The extension of the Weil-Petersson metric to the boundary of Teichm\"uller space.} 
Duke Math. J., 43 (1976), 623-635.

\bibitem{MilHom} Miller, E.
\emph{The homology of the mapping class group.}
Journal of Differential Geometry {\bf24} (1986), 1-14, Lehigh University.

\bibitem{MZoInv}  Manin, Y. and Zograf, P.
\emph{Invertible cohomological field theories and Weil-Petersson volumes.}
Ann. Inst. Fourier (Grenoble) {\bf 50} (2000), 519-535.

\bibitem{MZhCon} Mazzeo, R. and Zhu, X.
\emph{Conical metrics on Riemann surfaces, I: the compactified configuration space and regularity.}
Geom. \& Top. {\bf 24}. (2020), 309-372.

\bibitem{MirSim}
Mirzakhani, M. 
\emph{Simple geodesics and Weil-Petersson volumes of moduli spaces of bordered Riemann surfaces.}
Invent. Math. {\bf 167} (2007), 179-222.

\bibitem{MirWei}
Mirzakhani, M. 
\emph{Weil-Petersson volumes and intersection theory on the moduli space of curves.}
J. Amer. Math. Soc., {\bf 20} (2007), 1-23.

\bibitem{MonPoi} Mondello, G.
\emph{Poisson structures on the Teichm\"{u}ller space of hyperbolic surfaces with conical points.} 
In the tradition of Ahlfors-Bers. V, 307-329, Contemp. Math., {\bf 510}, Amer. Math. Soc., Providence, RI, 2010.

\bibitem{MonRie} Mondello, G.
\emph{Riemann surfaces with boundary and natural triangulations of the Teichm\"uller space.}
J. Eur. Math. Soc. {\bf 13} (2011), 635-684 .

\bibitem{MorCha}
Morita, S. 
\emph{Characteristic classes of surface bundles.}
Invent. Math. {\bf 90} (1987), 551-577.


\bibitem{MumRem} Mumford, D. 
\emph{A remark on Mahler's compactness theorem.} 
Proc. Amer. Math. Soc. {\bf 28} (1971), 289-294.

\bibitem{MumTow} Mumford, D. 
\emph{Towards an enumerative geometry of the moduli space of curves.} 
Arithmetic and Geometry: Papers Dedicated to IR Shafarevich on the Occasion of His Sixtieth Birthday. {\bf 2} (1983), 271-328, Springer.




\bibitem{NNAWei} N\"a\"at\"anen, M. and Nakanishi, T.
\emph{Weil-Petersson areas of the moduli spaces of tori.}
Result. Math. {\bf 33} (1998), 120-133.

\bibitem{NorCou}
Norbury, P.
\emph{Counting lattice points in the moduli space of curves.}
Math. Res. Lett. {\bf 17}, (2010), 467-481.

\bibitem{PenDec}
Penner, R.
\emph{The decorated Teichm{\"u}ller space of punctured surfaces.}
Comm. Math. Phys. {\bf 113} (1987), 299-339.

\bibitem{StiDeh} Stillwell, J.
\emph{The Dehn-Nielsen theorem, Appendix to "Papers on group theory and topology" by Max Dehn.}  
Springer Science \& Business Media (1988).

\bibitem{STrVar} Schumacher, G. and Trapani, S.
\emph{Variation of cone metrics on Riemann surfaces.}
J. Math. Anal. Appl.  {\bf 311} (2005), 218-230.

\bibitem{STrWei} Schumacher, G. and Trapani, S.
\emph{Weil-Petersson geometry for families of hyperbolic con\-ical Riemann surfaces.}
Michigan Math. J. {\bf 60} (2011), 3-33.

\bibitem{SWiJTG}
Stanford, D. and Witten, E.
\emph{JT gravity and the ensembles of random matrix theory.}
Adv. Theor. Math. Phys. {\bf 24}, (2020), 1475-1680.

\bibitem{TWZGen} Tan, S.P.; Wong, Y.L. and Zhang, Y.
\emph{Generalizations of McShane's identity to hyperbolic cone-surfaces.}
J. Diff. Geom. {\bf 72} (2006), 73-112.

\bibitem{WeiMod}
Weil, A.
\emph{Modules des surfaces de Riemann.} 
S\'eminaire Bourbaki. Expos\'e No. 168, (1958),  413-419.

\bibitem{WitQua} Witten, E.
\emph{On quantum gauge theories in two-dimensions.}
 Comm. Math. Phys. {\bf 141} (1991), 153-209.

\bibitem{WitTwo} Witten, E.
\emph{Two-dimensional gravity and intersection theory on moduli space.} Surveys in differential geometry (Cambridge, MA, 1990), 243--310, Lehigh Univ., Bethlehem, PA, 1991.

\bibitem{WolChe} 
Wolpert, S.
\emph{Chern forms and the Riemann tensor for the moduli space of curves.} 
Invent. Math. {\bf 85} (1986), 119-145.

\bibitem{WolEle} 
Wolpert, S.
\emph{An elementary formula for the Fenchel-Nielsen twist.} 
Comment. Math. Helv. {\bf 56} (1981), 132-135.

\bibitem{WolKah} Wolpert, S.
\emph{On the K{\"a}hler form of the moduli space of once punctured tori.}
Comment. Math. Helv. {\bf 58} (1983), 246-256.

\bibitem{WolNon} Wolpert, S.
\emph{Noncompleteness of the Weil-Petersson metric for Teichm\"uller space.} 
Pac. J. Math. {\bf 51} (1975), 513-576.

\bibitem{WolSym} Wolpert, S.
\emph{On the symplectic geometry of deformations of a hyperbolic surface.}
Ann. Math. (1983), 207-234. 

\bibitem{WolWei} 
Wolpert, S.
\emph{On the Weil-Petersson geometry of the moduli space of curves.}
Amer. J. Math. {\bf 107} (1985), 969-997. 


\end{thebibliography}
\end{document}